\documentclass{article}
\usepackage{amssymb}
\usepackage{amsmath}
\usepackage{pmat}

\newtheorem{first}{proposition}[section]

\newtheorem{lemma}[first]{Lemma}
\newtheorem{theorem}[first]{Theorem}
\newtheorem{proposition}[first]{Proposition}
\newtheorem{corollary}[first]{Corollary}

\makeatletter
\def\newrmtheorem#1{\@ifnextchar[{\@rmothm{#1}}{\@rmnthm{#1}}}

\def\@rmnthm#1#2{%
\@ifnextchar[{\@rmxnthm{#1}{#2}}{\@rmynthm{#1}{#2}}}

\def\@rmxnthm#1#2[#3]{\expandafter\@ifdefinable\csname #1\endcsname
{\@definecounter{#1}\@addtoreset{#1}{#3}%
\expandafter\xdef\csname the#1\endcsname{\expandafter\noexpand
  \csname the#3\endcsname \@rmthmcountersep \@rmthmcounter{#1}}%
\global\@namedef{#1}{\@rmthm{#1}{#2}}\global\@namedef{end#1}{\@endrmtheorem}}}

\def\@rmynthm#1#2{\expandafter\@ifdefinable\csname #1\endcsname
{\@definecounter{#1}%
\expandafter\xdef\csname the#1\endcsname{\@rmthmcounter{#1}}%
\global\@namedef{#1}{\@rmthm{#1}{#2}}\global\@namedef{end#1}{\@endrmtheorem}}}

\def\@rmothm#1[#2]#3{\expandafter\@ifdefinable\csname #1\endcsname
  {\global\@namedef{the#1}{\@nameuse{the#2}}%
\global\@namedef{#1}{\@rmthm{#2}{#3}}%
\global\@namedef{end#1}{\@endrmtheorem}}}

\def\@rmthm#1#2{\refstepcounter
    {#1}\@ifnextchar[{\@rmythm{#1}{#2}}{\@rmxthm{#1}{#2}}}

\def\@rmxthm#1#2{\@beginrmtheorem{#2}{\csname the#1\endcsname}\ignorespaces}
\def\@rmythm#1#2[#3]{\@opargbeginrmtheorem{#2}{\csname
       the#1\endcsname}{#3}\ignorespaces}

\def\@rmthmcounter#1{\noexpand\arabic{#1}}
\def\@rmthmcountersep{}
\def\@beginrmtheorem#1#2{\rm \trivlist
      \item[\hskip \labelsep{\bf #1\ #2\thmrmcounterend}]}
\def\@opargbeginrmtheorem#1#2#3{\rm \trivlist
      \item[\hskip \labelsep{\bf #1\ #2\ (#3)\thmrmcounterend}]}
\def\@endrmtheorem{\endtrivlist}
\def\thmrmcounterend{\hskip 0em\relax} 
\def\newrmwntheorem#1#2{\expandafter\@ifdefinable\csname #1\endcsname%
\global\@namedef{#1}{\@rmwnthm{#1}{#2}}%
\global\@namedef{end#1}{\@endrmwntheorem}}

\def\@rmwnthm#1#2{\@ifnextchar[{\@rmwnythm{#1}{#2}}{\@rmwnxthm{#1}{#2}}}
\def\@rmwnxthm#1#2{\@beginrmwntheorem{#2}\ignorespaces}
\def\@rmwnythm#1#2[#3]{\@opargbeginrmwntheorem{#2}{#3}\ignorespaces}
\def\@beginrmwntheorem#1{\rm \trivlist
      \item[\hskip \labelsep{\bf #1\thmrmwncounterend}]}
\def\@opargbeginrmwntheorem#1#2{\rm \trivlist
      \item[\hskip \labelsep{\bf #1\ (#2)\thmrmwncounterend}]}
\def\@endrmwntheorem{\endtrivlist}
\def\thmrmwncounterend{.\hskip 1em\relax}

\makeatother

\newrmtheorem{remark}[first]{Remark}

\newcommand{\rr}{\mathbb{R}}
\newcommand{\dd}{{\rm d}}
\newcommand{\ca}{\mathcal{A}}
\newcommand{\cb}{\mathcal{B}}

\newcommand{\cf}{\mathcal{F}}
\newcommand{\cg}{\mathcal{G}}

\newcommand{\ck}{\mathcal{K}}
\newcommand{\cl}{\mathcal{L}}
\newcommand{\cm}{\mathcal{M}}
\newcommand{\cn}{\mathcal{N}}
\newcommand{\cp}{\mathcal{P}}

\newcommand{\half}{\frac{1}{2}}
\newcommand{\vect}{\mathrm{vec}}

\newenvironment{proof}[1][\proofname]{\par \normalfont \trivlist
 \item[\hskip\labelsep\itshape #1]\ignorespaces
}{%
 \hspace*{\fill}$\Box$ \endtrivlist
}
\newcommand{\proofname}{{\bf Proof}}

\begin{document}

\pagestyle{empty}
\begin{center}
\Large
A block Hankel generalized confluent Vandermonde matrix  \bigskip\\
\large
Andr\'e Klein\footnote{Department of Quantitative Economics, Universiteit van Amsterdam, Roetersstraat 11,
1018WB  Amsterdam, {\tt a.a.b.klein@uva.nl}} and Peter Spreij\footnote{Korteweg-de Vries Institute for Mathematics,
Universiteit van Amsterdam, PO Box 94248, 1090GE Amsterdam,
{\tt spreij@uva.nl}}
\bigskip
\begin{abstract}
Vandermonde matrices are well known. They have a number of interesting properties and play a role in (Lagrange) interpolation problems, partial fraction expansions, and finding solutions to linear ordinary differential equations, to mention just a few applications.
Usually, one takes these matrices square, $q\times q$ say, in which case the $i$-th column is given by $u(z_i)$, where we write $u(z)=(1,z,\ldots,z^{q-1})^\top$. If all the $z_i$ ($i=1,\ldots,q$) are different, the Vandermonde matrix is non-singular, otherwise not. The latter case obviously takes place when all $z_i$ are the same, $z$ say, in which case one could speak of a confluent Vandermonde matrix. Non-singularity is obtained if one considers the matrix $V(z)$ whose $i$-th column ($i=1,\ldots,q$) is given by the $(i-1)$-th derivative $u^{(i-1)}(z)^\top$.

We will consider generalizations of the confluent Vandermonde matrix $V(z)$ by considering matrices obtained by using as building blocks the matrices $M(z)=u(z)w(z)$, with $u(z)$ as above and $w(z)=(1,z,\ldots,z^{r-1})$, together with its derivatives $M^{(k)}(z)$. Specifically, we will look at  matrices whose $ij$-th block is given by $M^{(i+j)}(z)$, where the indices $i,j$ by convention have initial value zero. These in general non-square matrices exhibit a block-Hankel structure. We will answer a number of elementary questions for this matrix. What is the rank? What is the null-space? Can the latter be parametrized in a simple way? Does it depend on $z$? What are left or right inverses? It turns out that answers can be obtained by factorizing the matrix into a product of other matrix polynomials having a simple structure. The answers depend on the size of the matrix $M(z)$ and the number of derivatives $M^{(k)}(z)$ that is involved. The results are obtained by mostly elementary methods, no specific knowledge of the theory of matrix polynomials is needed.
\medskip\\
{\sl keywords:} Hankel matrix, confluent Vandermonde matrix, matrix polynomial
\end{abstract}

\end{center}

\newpage
\pagestyle{plain}
\setcounter{page}{1}

\section{Introduction and notations}

\subsection{Motivation}

Vandermonde matrices are well known. They have a number of interesting properties and play a role in (Lagrange) interpolation problems, partial fraction expansions, and finding solutions to linear ordinary differential equations, to mention just a few applications. Usually, one takes these matrices square, $q\times q$ say, in which case the $i$-th column is given by $u(z_i)$, where we write $u(z)=(1,z,\ldots,z^{q-1})^\top$. If all the $z_i$ ($i=1,\ldots,q$) are different, the Vandermonde matrix is non-singular, otherwise not. The latter case obviously takes place when all $z_i$ are the same, $z$ say, in which case one could speak of a confluent Vandermonde matrix. Non-singularity is obtained if one considers the matrix $V(z)$ whose $i$-th column ($i=1,\ldots,q$) is given by the $(i-1)$-th derivative $u^{(i-1)}(z)^\top$, or by $u^{(i-1)}(z)^\top/(i-1)!$. In the latter case, one has $\det(V(z))=1$. A slightly more general situation is obtained, when one considers the in general non-square $q\times\nu$ matrix $V(z)$, with $i$-th column $u^{(i-1)}(z)^\top$, $i=1,\ldots,\nu-1$. In this case one has that $V(z)$ has rank equal to $\min\{q,\nu\}$ and for $\nu>q$, the kernel of $V(z)$ is the $(\nu-q)$-dimensional subspace of $\rr^\nu$ consisting of the vectors whose first $q$ elements are equal to zero. Note that the building elements of the matrix $V$ are the column vectors $u_i(z)$ and a number of its derivatives.

The observations above can be generalized in many directions. In the present paper we opt for one of them, in which we will consider generalizations of the confluent Vandermonde matrix $V(z)$ by considering matrices obtained by using as building blocks the matrices $M(z)=u(z)w(z)\in\rr^{q\times r}$, with $u(z)$ as above and $w(z)=(1,z,\ldots,z^{r-1})$, together with its derivatives $M^{(k)}$ with $0\leq k\leq \nu-1$. Note that $M(z)=V(z)$ if $r=1$.

A special case of what follows is obtained by considering the matrix
\[
\cm^0(z)=(M(z),\ldots,M^{(\nu-1)}(z))\in\rr^{q\times\nu r}.
\]
In a recent paper~\cite{ks2007}, the kernel of the matrix
$\cm^0(z)$ (or rather, a matrix obtained by a permutation of the columns of $\cm^0(z)$) has been studied for the case $q=r$ and $\nu\leq q$. 
More precisely, the two aims of the cited paper were to find a right
inverse and a parametrization of the kernel of the matrix $\cm^0(z)$. The two cases $\nu=q$ or $\nu< q$ have been analyzed in
detail. For these two cases two algorithms have been proposed to
construct the kernel and the rank and the dimension of the kernel
have been computed.
This was of relevance for the
characterization of a matrix polynomial equation having non-unique
solutions. The origin of this problem was of a statistical nature and lied in
properties of the asymptotic Fisher information matrix for estimating the parameters of an ARMAX
process and a related Stein equation.

The results on the parametrization of the  kernel obtained in~\cite{ks2007} turned out to
be unnecessary complicated, partly due an apparently irrelevant
distinction between different cases, and not transparent. In the
present paper we reconsider the problem by embedding it into a
much more general approach to obtaining properties of matrix
polynomials that can be viewed as generalizations of a confluent
Vandermonde matrix (in a single variable) as indicated above. The solution, which is
shown to exhibit a very simple and elegant structure, to the
original problem of describing the kernel under consideration, now
follows as a byproduct of the current analysis. These are special cases of the unified
situations of Corollary~\ref{cor:rankaz} and
Proposition~\ref{prop:kerkz} of the present paper. In this way we generalize in the
present paper the results of~\cite{ks2007}. Another approach to
find a basis for the kernel has been followed in~\cite{ks2005},
which is also subsumed by Proposition~\ref{prop:kerkz}. The results of \cite{ks2005} are closer in spirit to those of the present paper than the results in~\cite{ks2007}. Both cited papers contain some examples of right inverses of $\cm^0(z)$, that are special cases of what will be obtained in the present paper. Moreover, we will generalize the results for $\cm^0(z)$ to results for the matrix $\cm(z)$ that is defined by
\[
\cm(z)=
\begin{pmatrix}
M(z) & \cdots & M^{(\nu-1)}(z) \\
\vdots & & \vdots \\
M^{(\mu-1)}(z) & \cdots & M^{(\mu+\nu-2)}(z)
\end{pmatrix}
\in\rr^{\mu q\times \nu r},
\]
whose blocks we denote $\cm(z)^{ij}$, $i=0,\ldots,\mu-1$, $j=0,\ldots,\nu-1$, so $\cm(z)^{ij}=M^{(i+j)}(z)$.
We will also consider the related matrix $\cn(z)$ whose blocks are $\cn(z)^{ij}=\frac{M^{(i+j)}(z)}{i!j!}$.

A special case occurs for the choices of the parameters $r=1$, $\mu=1$ and $\nu=q$.
Then we write $\cm(z)=U_q(z)$ and $\cn(z)=\widetilde{U}_q(z)$, ordinary (normalized) $q\times q$ confluent Vandermonde matrices.
The results in the next sections (almost) reduce to trivialities for ordinary confluent
Vandermonde matrices, so the contribution of the present paper originates with allowing
the parameters $\mu$, $\nu$, $q$ and $r$ to be arbitrary. 

Classical Vandermonde matrices and confluent Vandermonde matrices have often been studied in the literature, see \cite{lt1985, fuhrmann1996} for definitions. More recent papers often focussed on finding formulas for their inverses and on efficient numerical procedures to compute them, a somewhat random choice of references are the papers \cite{chang1974, csaki1975, gohbergolshevsky1997, goknarizzet1973, kaufman1969, lupa1975, lutherrost2004, schapelle1972}. Nevertheless it seems that the generalization $\cm(z)$ of the confluent Vandermonde matrix is, to the best of our knowledge, unknown in the literature.

\subsection{Notations and conventions}

Derivatives of a function $z\mapsto f(z)$ (often matrix valued)
are denoted by $f^{(k)}(z)$ or by $(\frac{\dd}{\dd z})^k f(z)$.
The variable $z$ is in principle $\mathbb{C}$-valued. If $A$ is an
$m\times n$ matrix, for notational convenience we adopt the
convention to label its elements $A_{ij}$ with $i=0,\ldots,m-1$
and $j=0,\ldots,n-1$. We will see shortly why this a convenient
convention. Entries of a matrix are also indicated both by superindices, according to what is typographically most appropriate. When dealing with  matrix polynomials $A(z)$, in proving results we usually suppress the dependence on $z$ and simply write $A$.

If $\ca$ is a block matrix, we sometimes use subindexes
to indicate its constituting blocks, but more often denote
its blocks by superindices, so we write $\ca^{ij}$ or $A^{ij}$ in
the latter case. Also single superindices are used and we will
come across notations like $A^k$. These should not be confused
with powers of $A$. The meaning of $A^k$ will always be clear from
the context.

Throughout the paper $q,r,\mu$ and $\nu$ are fixed
positive integers, although often certain relations among them are
supposed (e.g.\ $q+1\leq \nu<q+r$). Given the integers $q,r,\mu$
and $\nu$ and the variable $z$, we consider  a number of matrix
polynomials. The first are $u(z)=(1,\ldots,z^{q-1})^\top$ and
$w(z)=(1,\ldots,z^{r-1})$. With the above labelling convention, valid throughout the paper, we
have  for the elements of $u(z)$ the expression $u_i(z)=z^i$ for
$i=0,\ldots, q-1$. Next to these we consider the  matrix
polynomials
\begin{align*}
M(z) & = u(z)w(z) \in\rr^{q\times r}, \\
\cm^0_j(z) & = M^{(j)}(z)= (\frac{\dd}{\dd z})^jM(z)\in\rr^{q\times r}, \\
\cm^0(z) & =
\begin{pmatrix} \cm^0_0(z),\ldots,\cm^0_{\nu-1}(z)
\end{pmatrix} \in\rr^{q\times \nu r}, \\
\cn^0_j(z) & = \frac{1}{j!}M^{(j)}(z)= \frac{1}{j!}(\frac{\dd}{\dd z})^jM(z)\in\rr^{q\times r}, \\
\cn^0(z) & =
\begin{pmatrix} \cn^0_0(z),\ldots,\cn^0_{\nu-1}(z)
\end{pmatrix} \in\rr^{q\times \nu r}.
\end{align*}
We also consider the $\mu q\times\nu r$ matrix
polynomial $\mathcal{M}(z)$ which is defined by its $ij$-blocks
\begin{equation}\label{eq:defcm}
\mathcal{M}^{ij}(z)= M^{(i+j)}(z),
\end{equation}
for $i=0,\ldots \mu-1$, $j=0,\ldots,\nu-1$.
The size of $\cm(z)$ is equal to $\mu
q\times\nu r$. If $\mu=1$, we retrieve $\cm^0(z)$. Note that $\cm(z)$ has a block Hankel structure, even for the case $\mu\neq \nu$.

Along with $\cm(z)$ we consider
the matrix $\cn(z)$ that has a block structure with blocks
\[
\cn^{ij}(z)=\frac{1}{i!j!}M^{(i+j)}(z),
\]
for $i=0,\ldots,\mu-1$ and $j=0,\ldots \nu-1$.  Note that $\cn(z)$
also has dimensions $\mu q\times \nu r$ and that $\cn(z)$ reduces to
$\cn^0(z)$ for $\mu=1$. Unlike $\cm(z)$, $\cn(z)$ doesn't exhibit a block Hankel structure. The following obvious relation holds.
\begin{equation}\label{eq:mdnd}
\cm(z)=(D_\mu\otimes I_q)\cn(z)(D_\nu\otimes I_r),
\end{equation}
where $D_\mu$ is the diagonal matrix with entries $D_{ii}=i!$, $i=0,\ldots,\mu-1$ and $D_\nu$ likewise.

Let us now introduce the  matrices $U_q(z)$ (of size $q\times q$) and $W_r(z)$ (of size $r\times r$) by
\[
U_q(z)=(u(z),\ldots,u^{(q-1)}(z))
\]
and
\[
W_r(z)^\top=(w(z)^\top,\ldots,w^{(r-1)}(z)^\top).
\]
Along with the matrices $U_q(z)$ and $W_r(z)$, we introduce the
matrices
\begin{equation}\label{eq:utilde}
\widetilde{U}_q(z)=(\widetilde{u}_0(z),\ldots,\widetilde{u}_{q-1}(z)),
\end{equation}
with
\[
\widetilde{u}_j(z)=\frac{1}{j!}u^{(j)}(z), \,j=0,\ldots,q-1,
\]
and $\widetilde{W}_r(z)$ given by
\begin{equation}\label{eq:wtilde}
\widetilde{W}_r(z)^\top =(\widetilde{w}_0(z),\ldots,\widetilde{w}_{r-1}(z)),
\end{equation}
with
\[
\widetilde{w}_j(z)=\frac{1}{j!}w^{(j)}(z), \,j=0,\ldots,r.
\]
We have the obvious relations
\[
U_q(z)=\widetilde{U}_q(z)D_q,\, W_r(z)=D_r\widetilde{W}_r(z),
\]
where $D_q$ and $D_r$ are diagonal matrices, similarly defined as $D_\mu$. One easily verifies that $\widetilde{U}_q(z)$ has elements
$\widetilde{U}_q^{ij}(z)=z^{i-j}{i \choose j}$ and by a simple
computation that $\widetilde{U}_q(z)^{-1}=\widetilde{U}_q(-z)$. Similar
properties hold for $\widetilde{W}_r(z)$. Later on we will also come across the matrices $U_\mu(z)$ and $W_\nu(z)$, which have the same structure as  $U_q(z)$ and $W_r(z)$ and only differ in size (unless $\mu=q$ and $\nu=r$).

Finally we introduce the often used square shift matrices $S_k\in\rr^{k\times k}$ ($k$ arbitrary), defined by its $ij$-elements $\delta_{i+1,j}$ (Kronecker deltas).
Other matrices will be introduced along the way.
\medskip\\
The structure of the remainder of the paper is as follows. After having fixed some notation and other conventions,
in Section~\ref{section:ca} we shall introduce an auxiliary matrix polynomial $\ca(z)$ that is instrumental in deriving properties of $\cm(z)$ and $\cn(z)$. Among them are factorizations, which will be treated in Section~\ref{section:factorization}. These lead to establishing the rank of $\cm(z)$, which is shown to be independent of $z$. A major issue in the present paper is to find a simple and transparent parametrization of the kernels of the matrices $\cm(z)$ and $\cn(z)$. This is done first in Section~\ref{section:kerneln0} for the special case of $\cn^0(z)$. The results of that section will be used in Section~\ref{section:kerneln} to characterize the kernel of $\cn(z)$.  The results of Section~\ref{section:intermezzo} form the cornerstone of finding right inverses of $\cm(z)$, especially those that are not dependent on $z$, which is a main topic of the final Section~\ref{section:rightinverse}.

\section{The matrix $\ca(z)$}\label{section:ca}
\setcounter{equation}{0}

This section is devoted to the matrix $\ca(z)$, to be defined below, that is instrumental in deriving properties of $\cm(z)$ and $\cn(z)$. In particular it is used for obtaining useful factorizations in Section~\ref{section:factorization}.
The  matrix $\ca(z)$, also of block Hankel type and of size $\mu q\times \nu r$
is defined by its blocks $\mathcal{A}_{ij}(z):=A^{i+j}(z)$ (here $i+j$ is
used as a super index) of size $q\times r$, for $i=0,\ldots
\mu-1$, $j=0,\ldots,\nu-1$. The matrices $A^k(z)$ for $k\geq 0$ are  the $q\times r$ quasi-symmetric matrix polynomials
(this matrix polynomial is only square for $q=r$) given by their elements
($i,j=0,\ldots,q-1$)
\begin{equation}\label{eq:akz}
A^k_{ij}(z)=\frac{1}{i!j!}(\frac{\dd}{\dd z})^{i+j}z^k={k \choose
i,j}z^{k-i-j}.
\end{equation}
Note that $k$ in $A^k(z)$ is used as a super index and in $z^k$ as
a power. The block Hankel matrix $\ca(z)$ of size $\mu q\times\nu r$ is
then defined by its blocks $\ca^{ij}(z)$ which are equal to
$A^{i+j}(z)$, for $i=0,\ldots,\mu-1$ and $j=0,\ldots,\nu-1$. One
easily verifies that (with $A^{k-1}_{-1,j}=A^{k-1}_{i,-1}=0$) for
$k\geq 1$ it holds that
\[
A^k_{ij}(z)=zA^{k-1}_{ij}(z)+A^{k-1}_{i-1,j}(z)+A^{k-1}_{i,
j-1}(z).
\]
In matrix notation this relation becomes
\begin{equation}\label{eq:Az1}
A^k(z)=zA^{k-1}(z)+A^{k-1}(z)S_r+S_q^\top A^{k-1}(z).
\end{equation}
Let $J_q(z)=zI_q+S_q^\top $. Then we can
rewrite (\ref{eq:Az1}) as
\begin{equation}\label{eq:A11}
A^k(z)=J_q(z)A^{k-1}(z)+A^{k-1}(z)S_r,
\end{equation}
or as
\begin{equation}\label{eq:A21}
A^k(z)=S_q^\top A^{k-1}(z)+A^{k-1}(z)J_r(z)^\top.
\end{equation}

By induction one easily proves that the recursion~(\ref{eq:A11})
leads to the explicit expressions, involving powers of $S_q^\top$ and $S_r$,
\begin{equation*}
A^k(z)=\sum_{i=0}^k{k \choose i}J_q(z)^iA^0S_r^{k-i}
\end{equation*}
and
\begin{equation}\label{eq:akl1}
A^{k+l}(z)=\sum_{i=0}^k{k \choose i}J_q(z)^iA^l(z)S_r^{k-i}.
\end{equation}
Likewise one shows that (\ref{eq:A21}) leads to
\begin{equation}\label{eq:ak01}
A^k(z)=\sum_{i=0}^k{k \choose i}(S_q^\top)^iA^0(J_r(z)^\top)^{k-i}.
\end{equation}
On the other hand, the recursion (\ref{eq:Az1}) has solution
\begin{equation}\label{eq:sol}
A^k(z)=\sum_{j=0}^k{k\choose j}z^{k-j}\sum_{i=0}^j{j\choose
i}(S_q^\top)^iA^0S_r^{j-i}.
\end{equation}
It also follows that
\[
A^k(0)=\sum_{i=0}^k{k\choose i}(S_q^\top)^iA^0S_r^{k-i}.
\]
Look back at the $A^k(z)$ as defined in~(\ref{eq:akz}). One
computes for $k\geq 0$
\begin{equation}\label{eq:A}
A^k_{ij}(0)=\left\{
\begin{array}{cl}
{k\choose i} & \mbox{ if }i+j=k\\
0 & \mbox{ else.}
\end{array}
\right.
\end{equation}
Henceforth we shall write $A^k$ instead of $A^k(0)$. Note that $A^{k}$, which
is in general not square, has only nonzero elements on one
`anti-diagonal', the one with $ij$-elements having sum $i+j=k$.
We have the
following result.
\begin{proposition}\label{prop:arec}
The matrices $A^k=A^k(0)$ satisfy the recursion for $k\geq 0$
\[
A^{k+1}=A^kS_r + S_q^\top A^k,
\]
and hence
\[
A^k=\sum_{j=0}^k{k\choose j}(S_q^\top)^jA^0S_r^{k-j}.
\]
Morever, we also have the polynomial expansion
\[
A^k(z)=\sum_{j=0}^k{k\choose j}z^{k-j}A^j.
\]
\end{proposition}
\begin{proof}
The recursion follows by taking $z=0$ in~(\ref{eq:Az1}). The expansion follows from Proposition~\ref{prop:arec} and Equation~(\ref{eq:sol}).
\end{proof}



\noindent
Two more matrices will be introduced next. First, let $\bar{\ca}$ be the $\mu q\times\nu r$ block-matrix whose $ij$-th
block (of size $q\times r$) is given by
$\bar{\ca}^{ij}=(S_q^\top)^jA^0S_r^i$ for $i=0,\ldots,\mu-1$ and
$j=0,\ldots,\nu-1$.

Second, we introduce the matrix $\cl_{\mu,q}(z)$ of size $\mu q\times\mu q$
with blocks $\cl_{\mu,q}(z)_{ij}={i \choose j}J_q(z)^{i-j}$ for
$i\geq j$ and zero else, $i,j=0,\ldots,\mu-1$. Since $\cl_{\mu,q}(z)$ is block lower diagonal with identity matrix as diagonal blocks, it follows that
$\cl_{\mu,q}(z)$ is invertible and that its inverse has $ij$-block
equal to ${i \choose j}J_q(z)^{i-j}(-1)^{i-j}$. Likewise one
defines the matrix $\cl_{\nu,r}(z)$ of size $\nu r\times\nu r$.

\begin{theorem}\label{thm:azfactor1}
The factorization
\begin{equation}\label{eq:lal1}
\ca(z) = \cl_{\mu,q}(z)\bar{\ca}\cl_{\nu,r}(z)^\top
\end{equation}
holds true. For $z=0$ and with $\ca=\ca(0)$ and $\cl_{\mu,q}=\cl_{\mu,q}(0)$ and
$\cl_{\nu,r}=\cl_{\nu,r}(0)$,  this becomes
\begin{equation}\label{eq:aa0}
\ca = \cl_{\mu,q}\bar{\ca}\cl_{\nu,r}^\top.
\end{equation}
Moreover, $\det(\ca(z))=\det(\bar{\ca})$, when both matrices are square.
\end{theorem}

\begin{proof}
We compute the $ij$-block on the right hand side of~(\ref{eq:lal1}).
Using the definitions of $\cl_{\mu,q}(z)$, $\cl_{\nu,r}(z)$ and $\bar{\ca}$,
we obtain (see the explanation below)
\begin{align*}
\big(\cl_{\mu,q}(z)\bar{\ca}\cl_{\nu,r}(z)^\top\big)_{ij} & = \sum_{k,l}\cl_{\mu,q}(z)_{ik}\bar{\ca}^{kl}\cl_{\nu,r}(z)^\top_{lj} \\
& = \sum_{k=0}^i\sum_{l=0}^j{i\choose k}J_q(z)^{i-k}(S_q^\top)^lA^0S_r^k{j\choose l}(J_r(z)^\top)^{j-l} \\
& = \sum_{k=0}^i{i\choose k}J_q(z)^{i-k}\big(\sum_{l=0}^j(S_q^\top)^lA^0{j\choose l}(J_r(z)^\top)^{j-l}\big) S_r^k \\
& = \sum_{k=0}^i{i\choose k}J_q(z)^{i-k}A^j(z) S_r^k \\
& = A^{i+j}(z).
\end{align*}
In the third equality above we used that $J_r(z)^\top$ and $S_r$ commute,
in the fourth equality we used~(\ref{eq:ak01}) with an appropriate change of the indices,
whereas the last equality similarly follows from~(\ref{eq:akl1}). The relation between the determinants follows from $\det(\cl_{\mu,q}(z))=1$.
\end{proof}
Let $e_i$ be the $i$-th standard column basis vector of $\rr^q$,
$i=0,\ldots,q-1$ and let $f_i$ be the $i$-th standard column basis vector of $\rr^r$,
$i=0,\ldots,r-1$. For convenience of notation, we put $e_i=0$ for $i\geq q$ and $f_j=0$ for $j\geq r$.
With this convention, we always have for example  $(S_q^\top )^ie_0=e_i$.

\begin{theorem}\label{thm:ao1}
It holds that the block $\bar{\ca}^{ij}=e_jf_i^\top$,
the rank of $\bar{\ca}$ is equal to $\min\{\mu,r\}\times\min\{\nu,q\}$.
If $\bar{\ca}$ is a fat matrix, i.e.\ $\mu q\leq \nu r$, then $\bar{\ca}$
has full (row)rank iff $\mu\leq r$ and $\nu\geq q$. If $\bar{\ca}$ is a tall matrix,
so $\mu q\geq \nu r$, then it has full (column) rank iff $\mu\leq r$ and $\nu\leq q$.
In the special case that $\bar{\ca}$ is square, so $\mu q=\nu r$, we get that $\bar{\ca}$ has full rank,
and it is then invertible, iff $\mu=r$ and $\nu=q$, in which case $\bar{\ca}$ is a matrix of size $qr\times qr$.
 In this case we have for the inverse $(\bar{\ca})^{-1}=(\bar{\ca})^\top$ and $\det(\bar{\ca})=(-1)^{\frac{1}{4}qr(q-1)(r-1)}$.
 Specializing even more to $q=r$, we get $\det(\bar{\ca})=(-1)^{\frac{1}{2}q(q-1)}$.
\end{theorem}
\begin{proof}
Since $A^0=e_0f_0^\top$ and $\bar{\ca}^{ij}=(S_q^\top)^jA^0S_r^i$, it follows that $\bar{\ca}^{ij}=f_je^\top_i$.

The rank of $\bar{\ca}$ is equal to the rank of $\bar{\ca}(\bar{\ca})^\top$, which is easy to compute.
We get for its
 $ij$-th block  $\sum_{l=0}^{\nu-1} e_lf_i^\top f_je_l^\top = f_i^\top f_j\sum_{l=0}^qe_le_l^\top =:f_i^\top f_j I_{q,\nu-1}$,
 where $I_{q,\nu-1}=\sum_{l=0}^{\nu-1} e_le_l^\top$. It follows that ${\rm rank}(I_{q,\nu-1})=\min\{q,\nu\}$.
 Furthermore we have $f_i^\top f_j=0$ for $i\neq j$ and $f_i^\top f_i=1$ iff $i\leq r-1$. We conclude that
  $\bar{\ca}(\bar{\ca})^\top$ is block diagonal, where the diagonal $ii$-blocks are equal to $I_{q,\nu-1}$ for $i\leq r-1$ and zero otherwise.
 The number of nonzero diagonal blocks is equal to $\min\{\mu,r\}$, hence the rank of
 $\bar{\ca}(\bar{\ca})^\top$
 is
 equal
 to
 $\min\{\mu,r\}\times\min\{\nu,q\}$.

Assume that $\bar{\ca}$ is fat and that $\mu\leq r$ and $\nu\geq q$.
Then the rank of $\bar{\ca}$ equals $\mu q$, the number of rows of
$\bar{\ca}$. For the converse statement we assume that $\mu> r$ (the
case $\nu < q$ can be treated similarly). In this case the rank
becomes $r\times \min\{\nu,q\}$, which is strictly less than $\mu
q$, the number of rows of $\bar{\ca}$, which then has rank deficiency.
The dual statements for a tall $\bar{\ca}$ follow by symmetry.

Assume next that the matrix $\bar{\ca}$ is square. It has full rank
iff the two sets of conditions for the tall and fat case hold
simultaneously, which yields the assertion on invertibility.
Assume then that $\mu=r$ and $\nu=q$. By the computations in the
first part of the proof we see that the diagonal blocks of
$\bar{\ca}(\bar{\ca})^\top$ are all equal to $I_q$. Since there are now
$r$ of these blocks, we obtain that $\bar{\ca}(\bar{\ca})^\top$ is the
$qr\times qr$ identity matrix.

To compute the determinant for this case, we observe that the
columns of $\bar{\ca}$ consist of all the basis vectors of $\rr^{qr}$,
but in a permuted order. Therefore, the determinant is equal to
plus or minus one. To establish the value of the sign, we compute
the number of inversions of the permutation. This turns out to be
equal to $\frac{1}{4}qr(q-1)(r-1)$, which results in a determinant
equal to $+1$ iff this number is even, and $-1$ in the other case.
If $r=q$, then $(-1)^{\frac{1}{4}q^2(q-1)^2}=(-1)^{\half q(q-1)}$.
A way of computing the number of inversions is to write the order
of the numbering of the column basis vectors in  rectangular
array. Decomposing every number  $x\in\{0,1,\ldots,qr-1\}$ in a
unique way as $x=nq+m$ with $m\in\{0,\ldots,q-1\}$ and
$n\in\{0,\ldots,r-1\}$, we can identify every such $x$ with a pair
$(m,n)$. An inversion $i(x,y)$ occurs when $x<y$, but in the
permuted order $x$ is preceded by $y$. For every $x$ the
number of inversions $i(x,y)$ is the number of elements in the
rectangle strictly to the South-West of $x$ in the rectangular
array. So, if $x$ corresponds to $(m,n)$ then the number of
inversions $i(x,y)$ is equal to $(q-1-m)n$. Summing these numbers
for $m=0,\ldots,q-1$ and $n=0,\ldots,r-1$ yields the total number
of inversions.
\end{proof}

\begin{remark}
In the case were $\bar{\ca}$ is square and invertible, it is the permutation matrix having the property $\vect(X^\top)= \bar{\ca}\vect(X)$, for any $X\in\rr^{q\times r}$.
\end{remark}

\begin{corollary}\label{cor:rankaz}
The rank of the matrix $\ca(z)$ is for all $z\in\mathbb{C}$ equal
to $\min\{\mu,r\}\times\min\{\nu,q\}$. If $\mu q=\nu r$, then
$\ca(z)$ is square and invertible and $\det
(\ca(z))=(-1)^{\frac{1}{4}qr(q-1)(r-1)}$.
\end{corollary}

\begin{proof}
The assertion on the rank follows from
Theorems~\ref{thm:azfactor1} and~\ref{thm:ao1} upon noting that
the matrices $\cl_{\mu,q}(z)$ and $\cl_{\nu,r}(z)$ are invertible.
Since $\det(\cl_{\mu,q}(z))=1$, also the assertion about the
determinant follows.
\end{proof}



\section{Factorizations of the matrices $\cm(z)$ and $\cn(z)$}\label{section:factorization}
\setcounter{equation}{0}

In the present section we obtain a factorization of the matrix polynomial $\cm(z)$, from which a factorization of the matrix $\cn(z)$ follows as a simple corollary.
In the next proposition we use the Kronecker symbol $\otimes$ to denote tensor products.
\begin{proposition}\label{prop:factorM}
The factorization
\begin{equation}\label{eq:m}
\mathcal{M}(z)=(I_{\mu}\otimes U_q(z))\mathcal{A}(I_{\nu}\otimes
W_r(z))
\end{equation}
holds true. Moreover, $\mathcal{M}(z)$ and $\mathcal{A}$ have
the same rank equal to $\min\{\mu,r\}\times\min\{\nu,r\}$.
\end{proposition}

\begin{proof} Computation of the product $(I_q\otimes
U_q(z))\mathcal{A}(I_q\otimes W_r(z))$ by using the definition of
$\mathcal{A}$ as block matrix,  yields a matrix that consists of
blocks
\[
U_q(z)A^kW_r(z)=\sum_{i=0}^{q-1}\sum_{j=0}^{r-1}A^k_{ij}u^{(i)}(z)w^{(j)}(z).
\]
Using Equation~(\ref{eq:A}), we get that this expression reduces
to
\begin{align*}
\sum_{i=0}^{q-1}\sum_{j=0}^{r-1}{k\choose i}
u^{(i)}(z)w^{(j)}(z)\delta_{j,k-i} & = \sum_{i=0\vee (k-r+1)}^{(q-1)\wedge k}{k\choose i}
u^{(i)}(z)w^{(k-i)}(z) \\
& = \sum_{i=0}^k {k\choose i}
u^{(i)}(z)w^{(k-i)}(z),
\end{align*}
since $u^{(i)}(z)=0$ for $i\geq q$ and $w^{(k-i)}(z)=0$ for $i\leq k-r$.
Recall that the blocks $\mathcal{M}(z)$ consist of the
derivatives $\frac{d^k}{dz^k}(u(z)w(z))$. The product rule for
differentiation then yields (\ref{eq:m}). The statement concerning the rank immediately follows from Theorem~\ref{thm:ao1}, since $U_q(z)$ and $W_r(z)$ are invertible.
\end{proof}

\begin{remark}
The factorization in Proposition~\ref{prop:factorM} exhibits a nice form of symmetry. Therefore it would be nice if also the matrix $\ca$ could be factorized in some symmetric way. There doesn't seem to be an easy way to do this. Consider a truly symmetric case, for instance $\mu=\nu=q=r=2$. In this case $\ca$ is non-singular and we have
\[
\ca=
\begin{pmatrix}
1 & 0 & 0 & 1 \\
0 & 0 & 1 & 0 \\
0 & 1 & 0 & 0 \\
1 & 0 & 0 & 2
\end{pmatrix}.
\]
A first reasonable factorization in this symmetric case would be of the form $\ca=AA^\top$, which would imply that $\ca$ is positive definite. One easily sees that this is not the case. As a next attempt, one could try to use the singular value decomposition of $\ca$, but the eigenvalues of $\ca$ are not particularly nice, so this look as a dead end too.
\end{remark}

\begin{remark}
Instead of the matrices $U_q(z)$ and $W_r(z)$, one can also use the
matrices $\widetilde{U}_q(z)$ (see~(\ref{eq:utilde})) and $\widetilde{W}_r(z)$
(see~(\ref{eq:wtilde})) to get a factorization of $\cm(z)$. One
then has to replace the matrix $\ca$ with $\widetilde{\ca}$, whose blocks $\widetilde{\ca}_{ij}$ are equal to
$\widetilde{A}^{i+j}$, where the matrices $\widetilde{A}^k$ are specified
by their elements
\begin{equation*}\label{eq:tildeA}
\widetilde{A}^k_{ij}=\left\{
\begin{array}{cl}
k! & \mbox{ if } i+j=k \\
0 & \mbox{ else. }
\end{array}
\right.
\end{equation*}
One then gets
\begin{equation}\label{eq:mm}
\mathcal{M}(z)=(I_{\mu}\otimes \widetilde{U}_q(z))\widetilde{\mathcal{A}}(I_{\nu}\otimes
\widetilde{W}_r(z)),
\end{equation}
which can be proved in the same way as~(\ref{eq:m}), or by using this identity and the relation
\begin{equation}\label{eq:dad}
\widetilde{A}^k=D_qA^kD_r.
\end{equation}
\end{remark}
Next we derive a factorization of the matrix
$\cn(z)$.
Let the matrix  $\hat{\ca}$ be defined by its $\mu\times \nu$ blocks $\hat{A}^{ij}$, for $i=0,\ldots \mu-1$,
$j=0,\ldots,\nu-1$, where the matrices $\hat{A}^{ij}\in \rr^{q\times r}$ have
elements

\begin{equation*}\label{eq:At}
\hat{A}^{ij}_{kl}=\left\{
\begin{array}{cl}
{i+j\choose i} & \mbox{ if } k+l=i+j \\
0 & \mbox{ else, }
\end{array}
\right.
\end{equation*}
for $k=0,\ldots, q-1$, $l=0,\ldots,r-1$.
\begin{proposition}\label{prop:factorN}
The factorization
\begin{equation}\label{eq:mt}
\cn(z)=(I_{\mu}\otimes
\widetilde{U}_q(z))\hat{\ca}(I_{\nu}\otimes \widetilde{W}_r(z))
\end{equation}
holds true.
The matrices $\ca$ and $\hat{\ca}$ are related through
\begin{equation}\label{eq:aahat}
(D_{\mu}\otimes I_q)\hat{\ca}(D_{\nu}\otimes I_r) = (I_{\mu}\otimes D_q)\ca (I_{\nu}\otimes D_r).
\end{equation}
Further relations between $\ca$, $\widetilde{\ca}$ and $\hat{\ca}$ are
\begin{align}
\widetilde{\ca} & =  (D_\mu\otimes I_q)\hat{\ca}(D_\nu\otimes I_r) \label{eq:aa1}\\
& =  (I_\mu\otimes D_q)\ca(I_\nu\otimes D_r)\label{eq:aa2}
\end{align}\end{proposition}
\begin{proof} The proof of the factorization~(\ref{eq:mt}) is similar to the proof of~(\ref{eq:m}).
Relation~(\ref{eq:aahat}) follows by an elementary computation. Likewise one proves~(\ref{eq:aa1}) and~(\ref{eq:aa2}).
\end{proof}




\begin{proposition}
Additionally we have the following factorizations of $\cm(z)$ and $\cn(z)$.
\begin{align}
\ca(z) & =(\widetilde{U}_\mu(z)\otimes I_q)\ca(0)(\widetilde{W}_\nu(z)\otimes I_r) \label{eq:aza0} \\
\mathcal{M}(z) & =(I_{\mu}\otimes \widetilde{U}_q(z))\mathcal{M}(0)(I_{\nu}\otimes
\widetilde{W}_r(z)) \label{eq:mzm0} \\
\cm(z) & =\big(\widetilde{U}_\mu(-z)\otimes U_q(z)\big)\ca(z)\big(\widetilde{W}_\nu(-z)\otimes W_r(z)\big) \label{eq:mzaz} \\
\cn(z) & =\big(U_\mu(z)^{-1}\otimes U_q(z)\big)\ca(z)\big(W_\nu(z)^{-1}\otimes W_r(z)\big) \label{eq:nzaz} \\
\mathcal{N}(z) & =(I_{\mu}\otimes \widetilde{U}_q(z))\mathcal{N}(0)(I_{\nu}\otimes
\widetilde{W}_r(z)) \label{eq:nzn0}
\end{align}

\end{proposition}
\begin{proof}
First we show an auxiliary result.
Let $\cp_{\mu,q}(z)=\cl_{\mu,q}(z)\cl_{\mu,q}(0)^{-1}$. Then
\begin{equation}\label{eq:pmuq}
\cp_{\mu,q}(z)=\widetilde{U}_\mu(z)\otimes I_q.
\end{equation}
Indeed, by direct computation starting from the definition of $\cl_{\mu,q}(z)$, one finds that the $ij$-block $\cp_{\mu,q}(z)^{ij}$ equals ${i\choose j}z^{i-j}I_q$ for $0\leq j\leq i\leq\mu-1$. Since $\widetilde{U}_\mu(z)_{ij}={i\choose j}z^{i-j}$, the result follows.

We now proceed to prove the indentities. Use (\ref{eq:lal1}), both for arbitrary $z$ and $z=0$, and Equation~(\ref{eq:pmuq}) to get (\ref{eq:aza0}). To obtain~(\ref{eq:mzm0}), in the same vain we use (\ref{eq:m}) twice. Then (\ref{eq:mzaz}) follows upon combining (\ref{eq:aza0}) and (\ref{eq:m}). Finally, (\ref{eq:nzaz}) follows from (\ref{eq:mzaz}) and (\ref{eq:mdnd}) and (\ref{eq:nzn0}) follows from (\ref{eq:mzm0}) and (\ref{eq:mdnd}).
\end{proof}
In the situation where $\cm(z)$ and $\cn(z)$ are square, we are able to compute their determinants.
\begin{corollary}\label{cor:dets}
Let $\mu q=\nu r$. Then all relevant matrices are square and we have the following expressions for the determinants of $\cm(z)$ and $\cn(z)$.
\begin{align*}
\det(\cm(z) & =(\prod_{i=0}^{q-1}i!)^{\mu}\,(\prod_{j=0}^{r-1}j!)^{\nu}\,\det(\ca), \\
\det(\cn(z)) & = \det(\hat{\ca}), \\
  & =(\prod_{i=0}^{\mu-1} i!)^q\,(\prod_{j=0}^{\nu-1} j!)^r\,\det(\ca),
\end{align*}
with $\det(\ca)$ given in Corollary~\ref{cor:rankaz}.
Hence the matrices $\cm(z)$ and $\cn(z)$ are unimodular.
\end{corollary}

\begin{proof}
The matrix $I_{\mu}\otimes\widetilde{U}_q(z)$ has determinant one, which
gives the first equation in view of~(\ref{eq:mt}).
Use~(\ref{eq:mm}), (\ref{eq:dad}) and its consequence
\[
\det(\widetilde{\ca})=(\det(D_q))^{\mu}(\det(D_r))^{\nu}\det(\ca)
\]
to get the formula for $\det(\cm(z))$. Then the last equation follows from~(\ref{eq:aahat}).
\end{proof}
In the next sections we investigate and characterize the kernel of
the matrix $\cn(z)$ for different values of the parameters.
It turns out that characterizing the kernel of $\cn(z)$ yields
more  elegant results than characterizing the kernel of $\cm(z)$.
On the other hand, a factorization of $\cm(z)$ leads to more
elegant expressions than factorizations of $\cn(z)$. Of course, in
view of~(\ref{eq:mdnd}), results for one of the two can easily be
transformed into results for the other.  We will focus on the
kernel of $\cn(z)$ only, simple because the obtained results have
a more elegant appearance. As it turns out, it is instrumental to consider the kernel of $\cn^0(z)$ (which is $\cn(z)$ for $\mu=1$) first, since the results obtained for this case, serve as building blocks for the kernel of $\cn(z)$.

\section{Characterization of the kernel of $\cn^0(z)$}\label{section:kerneln0}
\setcounter{equation}{0}

We will investigate the kernel of $\cn^0(z)$ for different values
of $\nu$,  $q$ and $r$. This is in different notation the problem
alluded to in the introduction and earlier investigated in~\cite{ks2007} for a special case. Notice that $\cn^0(z)$ is $q\times
r\nu$-dimensional. Furthermore for $\nu\leq q$  all derivatives
$u^{(k)}(z)$ ($k\leq \nu-1$) are nonzero and moreover linearly
independent vectors, whereas $u^{(k)}(z)\equiv 0$ for $k\geq q$.
It follows that the set $\mbox{span}\{u(z),\ldots,u^{\nu-1}(z)\}$ has
dimension equal to $\min\{\nu,q\}$. It is then an easy exercise to
directly see that the rank of $\cn^0(z)$ is also equal to $\min\{\nu,q\}$. Of course, this also follows from  Proposition~\ref{prop:factorN}.
In order to characterize the kernel of $\cn^0(z)$, we have to
discern three different cases. These are $\nu\leq q$, $q+1\leq \nu
< q+r$ and $\nu\geq q+r$, each case being treated in a separate
section. In all cases, our aim is to find  {\em simple, transparant} parametrizations of the kernels.

\subsection{The case $\nu\leq q$}\label{section:kern0}

Let $F(z)$ be the $r\times (r-1)$ matrix given by
\[
F(z)=
\begin{pmatrix}
-z     &   0     & \cdots & \cdots & 0       \\
 1     &  -z     & 0      & \cdots & 0     \\
 0     &   1     & \ddots & \ddots & \vdots         \\
\vdots & \ddots  & \ddots & \ddots & 0  \\
\vdots &         & \ddots & \ddots &    -z        \\
0      & \cdots  & \cdots & 0      &     1
\end{pmatrix}.
\]
Observe the trivial but crucial property, that the columns of
$F(z)$ span the $(r-1)$-dimensional null space of $w(z)$.
\medskip\\
Next we consider the matrix $K(z)$ that has the following block
structure. It  consists of matrices $K_{ij}(z)$ ($i=0,\ldots,\nu-1$,
$j=0,\dots,\nu-1$) where each $K_{ij}(z)$ has size $r\times (r-1)$
and is given by
\begin{align*}
K_{ii}(z) & = F(z),
\\
K_{i,i+1}(z) & = -F'(z),
\end{align*}
whereas all the other blocks are equal to zero. For $r=1$ the
matrix $K(z)$ is empty and by convention we say that the columns
of $K(z)$ in this case span the vector space $\{0\}$. Note that
$K(z)$ has dimensions $\nu r\times \nu(r-1)$.
\medskip\\
The matrix $K(z)$ looks as follows, where we suppress the
dependence on $z$ and omit zero blocks.
\begin{equation}\label{eq:defk}
K=I_\nu\otimes F + S_\nu\otimes F'=
\begin{pmatrix}
F & F' & & & \\
  & F & F' & & \\
& & \ddots & \ddots & \\
& & & \ddots & F' \\
& & & & F
\end{pmatrix},
\end{equation}
where $I_\nu$ is the $\nu$-dimensional unit matrix and $S_\nu$ the $\nu$-dimensional shift matrix, with elements $S_{ij}=\delta_{i+1,j}$.

\begin{proposition}\label{prop:kerkz}
Let $\nu\leq q$. The $\nu r\times \nu (r-1)$ matrix $K(z)$
has rank $\nu (r-1)$ and is such that $\cn^0(z)K(z)\equiv 0$.
In other words,  the columns of $K(z)$ form a basis for the kernel
of $\cn^0(z)$.
\end{proposition}
\begin{proof} Pick the $j$-th block column of $K(z)$, $K^j(z)$
say ($j\in\{0,\ldots,\nu-1\}$). We compute $\cn^0(z)K^j(z)$. For
$j=0$, this reduces to $M(z)F(z)$, which is zero, since
$w(z)F(z)=0$. Then, for $j\geq 1$ we have $ K_{j-1,j}(z)=F'(z)$.
Hence, for $j\geq 1$ we get (using that $F^{(k)}(z)=0$ for $k>1$,
\begin{align*}
\cn^0(z)K^j(z) & = \sum_{i=0}^{\nu-1} \cn^0_i(z)K_{ij}(z) \\
& = \frac{1}{j!} M^{(j)}(z)F(z) + \frac{1}{(j-1)!}M^{(j-1)}(z)F'(z)\\
& =  \frac{1}{j!} (M^{(j)}(z)F(z)+ j M^{(j-1) }(z)F'(z))\\
& = \frac{1}{j!} (M(z)F(z))^{(j)}\\
& = 0.
\end{align*}
Hence $\cn^0(z)K(z)=0$, so all columns of $K(z)$ belong to $\ker
\cn^0(z)$. Since we know that the rank of $\cn^0(z)$ is equal to
$\nu$, we get that $\dim \ker \cn^0(z)=
r\nu-\nu=(r-1)\nu$, which equals the number of columns
of $K(z)$. Since $K(z)$ is upper triangular with the full rank
matrices $F(z)$ on the block-diagonal, it has full rank.
Therefore, the columns of $K(z)$ exactly span the null space of
$\cn^0(z)$.
\end{proof}

\begin{remark} A nice feature of the matrix $K(z)$ in
Proposition~\ref{prop:kerkz}   is that
it is a matrix polynomial of degree $1$, only whereas the matrix
polynomial $\cn^0(z)$ has degree $q+r-2$.
\end{remark}

\subsection{The case $q+1\leq \nu < q+r$}

Next we consider what happens if $q+1\leq \nu \leq q+r-1$. Note that
this case is void if $r=1$. Henceforth we assume that $r\geq 2$. Let
the matrix $G(z)\in\rr^{r\times r}$ have elements
$G_{ij}(z)=z^{j-i-1}$ for $j>i$ and zero else. With $S=S_r$ the
$r$-dimensional shift matrix, we have the compact expression
$G(z)=S\sum_{k=0}^\infty(zS)^k$, since $S^k=0$ for $k\geq r$. We now
give some auxiliary results.

\begin{lemma}\label{lemma:g}
Let $G(z)$ be as above and let $S$ be the shifted $r\times r$ identity matrix,
$S_{ij}=\delta_{i+1,j}$. It holds that
\begin{align}
G(z) & =S(I-zS)^{-1} \label{eq:g1}\\
\frac{1}{j!}G^{(j)}(z) & = S^{j+1}(I-zS)^{-j-1}.\label{eq:g2} \\
\frac{1}{j!}G^{(j)}(z) & = G(z)^{j+1}\label{eq:g3}\\
\frac{1}{(j+1)!}G^{(j+1)}(z) & = \frac{1}{j!}G^{(j)}(z)G(z),
\label{eq:g4}
\end{align}
and the matrices $G^{(j)}(z)$ and $G^{(i)}(z)$ commute for all $i,j\geq 0$. One also has the
following two equivalent properties
\begin{align}
\frac{1}{j!}w^{(j)}(z) & = w(z)G(z)^j.\label{eq:w1} \\
\frac{1}{(j+l)!}w^{(j+l)}(z) & = \frac{1}{j!}w^{(j)}(z)G(z)^l.
\label{eq:w2}
\end{align}
Moreover,
\begin{align}\label{eq:wgk}
(w(z)G(z))^{(k)} & = (k+1)!\,w(z)G(z)^{k+1}.
\end{align}

\end{lemma}
\begin{proof} To prove~(\ref{eq:g1}), we recall that $S^j=0$ for
$j\geq r$. Hence
\[
G(z)=S\sum_{k=0}^\infty (zS)^k=S(I-zS)^{-1}.
\]
Then~(\ref{eq:g2}) simply follows by differentiation, which
immediately yields~(\ref{eq:g3}), from which (\ref{eq:g4}) also
follows.

We next prove~(\ref{eq:w1}) for $j=1$. This is then equivalent to
$w'(z)(I-zS)=w(z)S$, which can be verified by elementary
calculations. We proceed by induction and assume that
(\ref{eq:w1}) holds. Then we have
\begin{align*}
\frac{1}{(j+1)!}w^{(j+1)}(z) & = \frac{1}{(j+1)!}\frac{{\rm d}}{{\rm d}z}
w^{(j)}(z) \\
& = \frac{j!}{(j+1)!}\frac{{\rm d}}{{\rm d}z}(w(z)G(z)^j) \\
& = \frac{j!}{(j+1)!}(w'(z)G(z)^j+jw(z)G(z)^{j-1}G'(z)) \\
& = \frac{j!}{(j+1)!}(w(z)G(z)^{j+1}+jw(z)G(z)^{j+1}) \\
& = w(z)G(z)^{j+1}.
\end{align*}
This proves~(\ref{eq:w1}), which is easily seen to be equivalent
to~(\ref{eq:w2}). Equation~(\ref{eq:wgk}) follows by induction, using (\ref{eq:w1}).
\end{proof}

\begin{remark}
In fact, above statements about $G(z)$ also follow from the
parallel ones concerning $w(z)$ and the relation
\[
G(z)=
\begin{pmatrix}
w(z)S \\
\vdots \\
w(z)S^r
\end{pmatrix}.
\]
Indeed, since the matrices $G(z)$ and $S$ commute, one has for
instance
\[
\frac{1}{j!}w^{(j)}(z)S  = w(z)SG(z)^j,
\]
which together with similar relations leads to~(\ref{eq:g3}).
\end{remark}

\begin{lemma}\label{lemma:mug}
One has for $j\geq 1$
\begin{equation}\label{eq:mg}
\frac{1}{j!}(M^{(j)}(z)-u^{(j)}(z)w(z))=\frac{1}{(j-1)!}M^{(j-1)}(z)G(z).
\end{equation}
For $j\geq q$, this reduces to
\begin{equation}\label{eq:mgq}
\frac{1}{j!}M^{(j)}(z)=\frac{1}{(j-1)!}M^{(j-1)}(z)G(z).
\end{equation}
More generally, one has for $m,p\geq 0$
\begin{equation}\label{eq:mgpq}
\frac{M^{(q-1+p+m)}(z)}{(q-1+p+m)!}=\frac{M^{(q-1+m)}(z)}{(q-1+m)!}G(z)^p=\frac{M^{(q-1)}(z)}{(q-1)!}G(z)^{p+m}.
\end{equation}

\end{lemma}

\begin{proof}
Recall that $M(z)=u(z)w(z)$. Below we use~(\ref{eq:w2}) to compute
\begin{align*}
\frac{1}{j!}M^{(j)}(z) & = \frac{1}{j!}\sum_{i=0}^j{j\choose
i}u^{(j-i)}(z)w^{(i)}(z) \\
& = \sum_{i=0}^j\frac{u^{(j-i)}(z)}{(j-i)!}\frac{w^{(i)}(z)}{i!} \\
& = \sum_{i=1}^j\frac{u^{(j-i)}(z)}{(j-i)!}\frac{w^{(i-1)}(z)}{(i-1)!}G(z) + \frac{u^{(j)}(z)}{j!}w(z) \\
& = \frac{1}{(j-1)!}\sum_{i=1}^j{j-1\choose i-1}u^{(j-i)}(z)w^{(i-1)}(z)G(z) + \frac{u^{(j)}(z)}{j!} w(z)\\
& = \frac{1}{(j-1)!}\sum_{l=0}^{j-1}{j-1\choose l}u^{(j-1-l)}(z)w^{(l)}(z)G(z) + \frac{u^{(j)}(z)}{j!}w(z) \\
& = \frac{1}{(j-1)!}M^{(j-1)}(z)G(z) + \frac{u^{(j)}(z)}{j!}w(z).
\end{align*}
Equation~(\ref{eq:mgq}) follows from (\ref{eq:mg}), because $u^{(j)}=0$ for $j\geq q$. Finally, (\ref{eq:mgpq}) follows by iteration of (\ref{eq:mgq}).
\end{proof}
We consider the
matrix polynomial $\bar{K}(z)$ which for the present case has the
following structure. For $i,j=0,\ldots, q-1$ it has blocks
$\bar{K}_{ij}(z)$ of size $r\times (r-1)$ such that $\bar{K}_{jj}(z)=F(z)$ and
$\bar{K}_{j-1,j}(z)=F'(z)$, $F(z)$ as before. For $i,j=q,\ldots,\nu-1$ the
blocks are $\bar{K}_{jj}(z)=I_r$ and $\bar{K}_{j-1,j}(z)=-G(z)$, all of size
$r\times r$. Finally, we have that $\bar{K}_{q-1,q}(z)=-G(z)$. All other
blocks are equal to zero. Notice that $\bar{K}(z)$ is of size
$\nu r\times (\nu r-q)$. One easily verifies that $\bar{K}(z)$ has
full column rank.
\medskip\\
The matrix $\bar{K}(z)$ looks as follows, where again we
suppress the dependence on $z$ and omit zero blocks.
\begin{equation}\label{eq:kbar}
\bar{K}=
\pmatnocross\begin{pmat}({....|....})
F & F' & & & & & & & & \cr
& F & F' & & & & & & & \cr
& & \ddots & \ddots & & & & & & \cr
& & & \ddots & F' & & & & & \cr
& & & & F & -G & & & & \cr\-
& & & & & I_r & -G & & & \cr
& & & & & & I_r & -G & & \cr
& & & & & & & \ddots & \ddots & \cr
& & & & & & & & \ddots & -G \cr
& & & & & & & & & I_r \cr
\end{pmat}.
\end{equation}
A compact expression for $\bar{K}$ is
\begin{equation}\label{eq:bark}
\bar{K}=
\begin{pmatrix}
I_q\otimes F + S_q\otimes F' & -\ell_q \mathrm{f}_{\nu-q}^\top \otimes G \\
0 & I_{\nu-q}\otimes I_r - S_{\nu-q}\otimes G
\end{pmatrix},
\end{equation}
where $\mathrm{f}_{\nu-q}$ is the first standard basis vector of $\rr^{\nu-q}$ and $\ell_q$ the last standard basis vector of $\rr^q$.

\begin{proposition}\label{prop:kerkz1}
Let $q+1\leq \nu \leq  q+r-1$. The $\nu r\times (\nu r- q)$ matrix
$\bar{K}(z)$ is such that $\cn^0(z)\bar{K}(z)\equiv 0$. In other words,  the
columns of $\bar{K}(z)$ form a basis for the kernel of $\cn^0(z)$.
\end{proposition}
\begin{proof}  Next we proceed as in the proof of
Proposition~\ref{prop:kerkz}. We pick the $j$-block column
$\bar{K}^j(z)$. If $0\leq j\leq q-1$, then $\bar{K}_{jj}(z)=F(z)$ and
$\bar{K}_{j-1,j}(z)=F'(z)$, whereas all other $\bar{K}_{ij}(z)$ are zero. The
computation of $\cn^0(z)\bar{K}^j(z)$ is then exactly as in the previous
proof. Next we consider the case where $q \leq j\leq \nu-1$, which
is quite different. Again we pick the $j$-th block column of
$\bar{K}(z)$. Recall the definition of the $\bar{K}_{ij}(z)$ for this case. We
get,
\begin{align*}
\cn^0(z)\bar{K}^j(z) & = \sum_{i=0}^{\nu-1} \cn^0_{i}(z)\bar{K}_{ij}(z) \\
& = \frac{1}{j!}M^{(j)}(z)-\frac{1}{(j-1)!}M^{(j-1)}(z)G(z) \\
& = 0,
\end{align*}
in view of Equation~(\ref{eq:mgq}). This shows that $\bar{K}(z)$ belongs to the kernel of $\cn^0(z)$.
Since $\cn^0(z)$ has rank $q$, the dimension of the kernel is
equal to $\nu r-q$, which is equal to the rank of $\bar{K}(z)$. Hence
the columns of $\bar{K}(z)$ span this kernel.
\end{proof}
Post-multiplying the matrix $\bar{K}(z)$ by
\[
\pmatnocross\begin{pmat}({....|....})
I_{r-1} & 0 & & & & & & & & \cr
& I_{r-1}& 0 & & & & & & & \cr
& & \ddots & \ddots & & & & & & \cr
& & & \ddots & 0 & & & & & \cr
& & & & I_{r-1} & & & & & \cr\-
& & & & & I_r & G & \cdots & \cdots & G^{\nu-q-1}\cr
& & & & & & I_r & G & \cdots & G^{\nu-q-2}\cr
& & & & & & & \ddots & \ddots & \cr
& & & & & & & & \ddots & G \cr
& & & & & & & & & I_r \cr
\end{pmat},
\]
we obtain
\[
\pmatnocross\begin{pmat}({....|....})
F & F' & & & & & & & & \cr
& F & F' & & & & & & & \cr
& & \ddots & \ddots & & & & & & \cr
& & & \ddots & F' & & & & & \cr
& & & & F & -G & -G^2 & \cdots & \cdots & -G^{\nu-q}\cr\-
& & & & & I_r &  & & & \cr
& & & & & & I_r & & & \cr
& & & & & & & \ddots &  & \cr
& & & & & & & & \ddots &  \cr
& & & & & & & & & I_r \cr
\end{pmat},
\]
an alternative matrix whose columns span $\ker\cn^0(z)$.

\subsection{The case $\nu\geq q+r$}

For this case, the kernel of $\cn^0(z)$ is closely related to what we have obtained in the previous case.

\begin{proposition}\label{prop:kerkz2}
Let $\nu \geq q+r$. Consider the matrix $\bar{K}_*(z)$ of~(\ref{eq:kbar}) in the
special case that $\nu=q+r-1$, then $\bar{K}_*(z)$ has $(q+r)(r-1)$
columns. We have that $\ker \cn^0(z)$, which is now
$\nu r-q$-dimensional, is the product of  the space spanned by
the columns of $\bar{K}_*(z)$ and $\mathbb{R}^{r(\nu+1-q-r)}$.
\end{proposition}
\begin{proof} Since the highest power of $z$ that
appears in $\cn^0(z)$ is $q+r-2$, we have that the matrices
$M^{(j)}(z)$ are identically zero if $\nu\geq q+r-1$, whereas the
matrix $\big(\frac{1}{0!}M(z),\ldots,\frac{1}{(q+r-2)!}M^{(q+r-2)}(z)\big)$ is
the same as the matrix $\cn^0(z)$ for the case $\nu=q+r-1$. The
assertion follows by application of Proposition~\ref{prop:kerkz1}.
\end{proof}

\section{Characterization of the kernel of $\cn(z)$}\label{section:kerneln}
\setcounter{equation}{0}

We have seen that we had to distinguish three different cases to
describe the kernel of $\cn^0(z)$. The same distinction has to be
made in the present section.

\subsection{The case $\nu\leq q$}\label{section:nlq}

First we introduce some more notation. For the matrix $K(z)$ of
(\ref{eq:defk}) we now write $K_0(z)$ and for $k=1,\ldots,r-2$, we
define $K_k(z)$ in the same way as $K(z)$, but now with $F$
replaced with $F_k(z)$, a $(r-k)\times(r-1-k)$ matrix having the
same structure as the original $F(z)$ of Section~\ref{section:kern0}, so $F_k(z)_{ij}=-z$, if $i=j$ and $F_k(z)_{ij}=1$, if $i=j+1$. In particular $F_0(z)=F(z)$. Formally, for $1\leq j\leq r-2$, we have
\[
F_k(z)= \begin{pmatrix}I_{r-k} & 0_{(r-k)\times k}\end{pmatrix}F(z)\begin{pmatrix}I_{r-k-1} \\ 0_{k\times (k-j-1)}\end{pmatrix},
\]
whereas $F_k(z)$ is the empty matrix for $k\geq r-1$. By $\ck^{\mu-1}(z)$ for
$\mu< r$, we denote the product $K_0(z)K_1(z)\cdots K_{\mu-1}(z)$ of size $\nu r\times \nu(r-\mu)$,
whereas we take $\ck^{\mu-1}(z)$ the   zero matrix for $\mu\geq r$.

For $j=0,\ldots,r-1$ we put
\[
w_j(z)=w(z)\begin{pmatrix}
I_{r-j} \\
0_{j\times(r-j)}
\end{pmatrix}
=(1,z,\ldots,z^{r-1-j})
\]
and $M_j(z)=u(z)w_j(z)$. With this convention, we have $w_0=w$, $M_0=M$.

\begin{lemma}
Let
$w_1(z)=(1,z,\ldots,z^{r-2})$ and
$M_1(z)=u(z)w_1(z)$. For $j\geq 1$ one has
\begin{align}
w^{(j)}(z)F(z) & = -jw^{(j-1)}(z)F'(z)=jw_1^{(j-1)}(z)\label{eq:wf}\\
M^{(j)}(z)F(z) & = -jM^{(j-1)}(z)F'(z)=jM_1^{(j-1)}(z).\label{eq:mff}
\end{align}
\end{lemma}

\begin{proof}
Since $w(z)F(z)=0$, also
$\frac{{\rm d}^j}{{\rm d}z^j}(w(z)F(z))=0$ for all $j\geq 1$. It then follows
that $w^{(j)}(z)F(z)+jw^{(j-1)}(z)F'(z)=0$, since all higher order
derivatives of $F$ vanish. Using $-w(z)F'(z)=w_1(z)$, we arrive at (\ref{eq:wf}). Similarly, one obtains (\ref{eq:mff}).
\end{proof}

\begin{theorem}\label{thm:kernel1}
The kernel of $\cn(z)$ is spanned by the columns of
$\ck^{\mu-1}(z)$ and its dimension is equal to
$\nu(r-\mu)^+$. Hence the rank of $\cn(z)$ is equal to
$\nu\min\{\mu,r\}$. In particular the matrix $\mathcal{N}(z)$ has
full rank iff $\mu\geq r$.
\end{theorem}

\begin{proof}
The arguments used in the proof of Proposition~\ref{prop:kerkz}
can also be applied to this more general case. Let for $k\geq 0$
\[
\cn^k(z)=\frac{1}{k!}\big(\frac{1}{0!}M^{(k)},\ldots,\frac{1}{(\nu-1)!}M^{(k+\nu-1)}(z)\big).
\]
Proposition~\ref{prop:kerkz} yields $\cn^0(z)K_0(z)=0$, and therefore $\cn^0(z)\ck^{\mu-1}(z)=0$. Consider now $k\geq 1$. As before, $K^j(z)$ denotes the $j$-th block-column of $K_0(z)=K(z)$.
Then, from~(\ref{eq:mff}) it follows that 
\[
\cn^k(z)K^0(z)=\frac{1}{k!}M^{(k)}(z)F(z)=\frac{1}{(k-1)!}
M_1^{(k-1)}(z). 
\]
For $j\geq 1$ we get
\begin{align*}
\cn^k(z)K^j(z) &
=\frac{1}{k!}\big(\frac{1}{(j-1)!}M^{(k+j-1)}(z)F'(z)+\frac{1}{j!}M^{(k+j)}(z)F(z)\big)
\\ & = \frac{1}{(k-1)!j!}M_1^{(k+j-1)}(z),
\end{align*}
where we used~(\ref{eq:mff}) again. It follows that for $k\geq 1$
\begin{equation}\label{eq:mbark}
\cn^k(z)K_0(z)=\cn_1^{k-1}(z),
\end{equation}
where
\[
\cn_1^{k-1}(z)=\frac{1}{(k-1)!}\big(\frac{1}{0!}M_1^{(k-1)},\ldots,\frac{1}{(\nu-1)!}M_1^{(k+\nu-2)}(z)\big),
\]
which is a matrix of size $ q \times \nu(r-1)$. Invoking Proposition~\ref{prop:kerkz} again, we obtain for $k=1$
\[
\cn^1(z)K_0(z)K_1(z)=\cn_1^{0}(z)K_1(z)=0,
\]
and hence $\cn^1(z)\ck_{\mu-1}(z)=0$. Assume that
$\mu<r$. To have $\cn(z)\ck^{\mu-1}(z)=0$, we need
$\cn^k(z)\ck^{\mu-1}(z)=0$ for $k=0,\ldots,\mu-1$. This now follows by
iteration of (\ref{eq:mbark}). In fact, by induction, one can show
\begin{equation}\label{eq:nkj}
\cn^k(z)K_0(z)\cdots K_{j-1}(z)=\cn^{k-j}_j(z)\mbox{ for }j\leq k,
\end{equation}
where
\[
\cn^{k-j}_j(z)= \frac{1}{(k-j)!}\big(\frac{1}{0!}M_j^{(k-j)},\ldots,\frac{1}{(\nu-1)!}M_j^{(k+\nu-j-1)}(z)\big).
\]
For $j\geq k+1$ one has $\cn^k(z)K_0(z)\cdots K_{j-1}(z)=0$.

Since each of the matrices $K_k(z)$ for
$k< r-1$ has full rank, which is equal to $\nu(r-k-1)$, we get
that $\ck^{\mu-1}(z)$ has rank equal to $\nu(r-\mu)$. All
assertions for $\mu< r$ now follow. On the other hand, for
$\mu\geq r$, the matrix $\cn(z)$ has rank equal to $\nu r$, and
therefore has a zero kernel.
\end{proof}

\begin{remark}
The matrix $K^{\mu-1}(z)$ for $\mu\leq r-1$, which is of size $\nu
r\times \nu(r-\mu)$ turns out to be upper block-triangular.
Consider for this case the product $\cf_{\mu-1}(z)=F_0(z)\cdots
F_{\mu-1}(z)$. Then one has for $j\geq i$ the $ij$-block
\[
\ck^{\mu-1}(z)_{ij}=\frac{1}{(j-i)!}\cf_{\mu-1}^{(j-i)}(z),
\]
which can easily be shown by induction.
\end{remark}





\subsection{The case $q+1\leq \nu <q+r$}\label{section:2ndcase}

Next we extend the result of Proposition~\ref{prop:kerkz1} to
obtain the kernel of the matrix $\cn(z)$ for the present case. The
approach that we follow is the same as the one leading to
Theorem~\ref{thm:kernel1}.

To obtain our results,  we need to introduce additional notation.
Let $G_{kj}(z)$  denote the upper left block of $G(z)$ having size $(r-k)\times (r-j)$ for $0\leq k,j\leq r-1$. So
\[
G_{kj}(z)=\begin{pmatrix}I_{r-k} & 0_{(r-k)\times k}\end{pmatrix}G(z)\begin{pmatrix}I_{r-j} \\ 0_{j\times (r-j)}\end{pmatrix}.
\]
We also need the matrices $\cg_0(z)=I_r$, and $\cg_{i}(z)=G_{i,i-1}(z)\cdots G_{10}(z)\in\rr^{(r-i)\times r}$, for $0<i\leq r-1$.

\begin{lemma}\label{lemma:gjik}
It holds that $G_{jj}(z)G_{j,j-1}(z)=G_{j,j-1}(z)G_{j-1,j-1}(z)$ for $j\geq 1$ and for $j>i>k\geq 0$ one has $G_{ji}(z)G_{ik}(z)=G_{j,i-1}(z)G_{i-1,k}(z)=G_{jk}(z)G_{kk}(z)$.
\end{lemma}
\begin{proof}
The first assertion follows from the decomposition
\[
G_{j-1,j-1}  =\begin{pmatrix}G_{j,j-1} \\0 \end{pmatrix}=\begin{pmatrix}G_{jj} & g \\ 0 & 0\end{pmatrix},
\]
where $g$ is the last column of $G_{j,j-1}$. For the proof of the second assertion we need the following property of the shift matrix $S\in\rr^k$ ($k$ according to the context): $\widetilde{I}S=S$, where
\[
\widetilde{I}=\begin{pmatrix}I_{k-1} & 0 \\0 & 0\end{pmatrix}.
\]
Since any $G_{ii}(z)$ is of the form $S(I-zS)^{-1}$ (with $S$ of size $(r-i)\times (r-i)$, see Lemma~\ref{lemma:g}), we have $I_0G_{ii}(z)=G_{ii}(z)$.
Then we compute
\begin{align*}
G_{ji}G_{ik} & =\begin{pmatrix}I & 0 \end{pmatrix}G_{i-1,i-1}\begin{pmatrix}I \\ 0\end{pmatrix}\begin{pmatrix}I & 0 \end{pmatrix}G_{i-1,i-1}\begin{pmatrix}I \\ 0\end{pmatrix} \\
& = \begin{pmatrix}I & 0 \end{pmatrix}G_{i-1,i-1}\begin{pmatrix}I & 0 \\0 & 0 \end{pmatrix}G_{i-1,i-1}\begin{pmatrix}I \\ 0\end{pmatrix} \\
& = \begin{pmatrix}I & 0 \end{pmatrix}G_{i-1,i-1}\widetilde{I}G_{i-1,i-1}\begin{pmatrix}I \\ 0\end{pmatrix} \\
& = G_{j,i-1}G_{i-1,k}.
\end{align*}
\end{proof}
\begin{lemma}\label{lemma:ggg} For $i\geq 1$,
it holds that
\begin{align}
\frac{1}{i!}G_{i,0}^{(i)}(z)& =G_{i0}(z)G(z)^i, \nonumber\\
\cg_{i}(z) & = G_{i0}(z)G(z)^{i-1} \label{eq:ggg2}\\
G_{ii}(z)\cg_i(z)& =\cg_i(z)G(z).\label{eq:ggg3}
\end{align}
\end{lemma}
\begin{proof}
Using the definition of $G_{i0}$, the equality $\frac{1}{i!}G_{i,0}^{(i)}=G_{i0}G^i$ immediately follows from (\ref{eq:g3}).
The second equality (\ref{eq:ggg2}) is obviously true for $i=0$. We use
induction. Let $i\geq 1$ and assume that $\cg_{i}=G_{i0}G^{i-1}$. Then, using Lemma~\ref{lemma:gjik}, $\cg_{i+1}G=G_{i+1,i}\cg_iG=G_{i+1,i}G_{i0}G^i=G_{i+1,0}G^{i+1}$. To prove (\ref{eq:ggg3}), we use (\ref{eq:ggg2}) and Lemma~\ref{lemma:gjik} to write
$
G_{ii}\cg_i=G_{ii}G_{i0}G^i=G_{i0}GG^i=\cg_iG.
$
\end{proof}
\begin{lemma}\label{lemma:mqpm}
Let $w_j(z)$ and $M_j(z)$ be as in  Section~\ref{section:nlq}. It holds that
\begin{align}
w_j^{(k)}(z)\cg_j(z) & = w^{(k)}G^j(z) \label{eq:wgkj1} \\
\frac{1}{k!}w_j^{(k)}(z)\cg_j(z) & =\frac{1}{(k+j)!}w^{(k+j)}(z) \label{eq:wgkj2} \\
M_j^{(k)}(z)\cg_j(z) & = M^{(k)}(z)G(z)^j.\label{eq:mgkj}
\end{align}
Moreover, for $k\geq 0$ it holds that
\begin{equation}\label{eq:mqq-1}
\frac{M_j^{(q-1+k)}(z)}{(q-1+k)!}\cg_j(z)=\frac{1}{(q-1)!}M^{(q-1)}(z)G^{k+j+1}.
\end{equation}
\end{lemma}
\begin{proof}
We need the following observation. For a row vector $x$ of appropriate length and a scalar $y$, one has
\begin{equation}\label{eq:gxy}
(x,y)G_{jj}=x\,G_{j,j-1},
\end{equation}
because \[
G_{jj}=\begin{pmatrix}
G_{j,j-1} \\ 0
\end{pmatrix}.
\]
We now show (\ref{eq:wgkj1}).  It is obviously true for $j=0$. Assume it holds for some $j\geq 0$. We get, using (\ref{eq:ggg3}) and (\ref{eq:gxy})
\begin{align*}
w^{(k)}_{j+1}\cg_{j+1} & = w^{(k)}_{j+1}G_{j+1,j}\cg_{j} \\
& = w^{(k)}_{j}G_{jj}\cg_{j} \\
& = w^{(k)}_{j}\cg_{j} G\\
& = w^{(k)}G^{j+1}.
\end{align*}
Equation~(\ref{eq:wgkj2}) follows by combining (\ref{eq:wgkj1}) and (\ref{eq:w2}), whereas (\ref{eq:mgkj}) is an immediate consequence of (\ref{eq:wgkj1}).
Next we compute, using (\ref{eq:mgkj}) and (\ref{eq:mgpq}),
\begin{align*}
\frac{M_j^{(q-1+k)}(z)}{(q-1+k)!}\cg_j(z) & = \frac{M^{(q-1+k)}(z)}{(q-1+k)!}G(z)^j \\
& = \frac{M^{(q-1)}(z)}{(q-1)!}G(z)^{k+j+1},
\end{align*}
which yields (\ref{eq:mqq-1}).
\end{proof}
Let for $j\leq r-2$ the matrix $\bar{K}_j(z)$ be given by

\begin{equation}\label{eq:kbarj}
\bar{K}_j=
\pmatnocross\begin{pmat}({....|....})
F_j & F_j' & & & & & & & & \cr
& F_j & F_j' & & & & & & & \cr
& & \ddots & \ddots & & & & & & \cr
& & & \ddots & F_j' & & & & & \cr
& & & & F_j & -\frac{1}{j!}G_{j0}^{(j)} & & & & \cr\-
& & & & & I & -G & & & \cr
& & & & & & I & -G & & \cr
& & & & & & & \ddots & \ddots & \cr
& & & & & & & & \ddots & -G \cr
& & & & & & & & & I \cr
\end{pmat}.
\end{equation}
Here $\bar{K}_j$ has $q$ diagonal entries $F_j$ and $\nu-q$ diagonal entries $I=I_r$. Hence $\bar{K}^j$ has dimensions $(\nu r-qj)\times(\nu r - q(j+1))$.
A compact expression of $\bar{K}^j$ is as follows. Let $\ell_q$ be the last standard basis vector of $\rr^q$, $\mathrm{f}_{\nu-q}$ the first basis vector of $\rr^{\nu-q}$, and $S_q$ the shift matrix of size $q\times q$.
Then, similar to (\ref{eq:bark}),
\[
\bar{K}_j=\begin{pmatrix}
I_q\otimes F_j+S_q\otimes F_j' & -\ell_q \mathrm{f}_{\nu-q}^\top\otimes \frac{1}{j!} G_{j0}^{(j)} \\
0 & I_{\nu-q}\otimes I_r - S_{\nu-q}\otimes G
\end{pmatrix}.
\]
Note that the matrices $F_j$ are empty for $j\geq r-1$,
$\frac{1}{(r-1)!}G_{r-1,0}^{(r-1)}(z)=(0,\ldots,0)$ and that
$G_{j0}^{(j)}$ is empty for $j\geq r$. Hence we define
\begin{equation}\label{eq:kbarr-1}
\bar{K}_{r-1}=
\pmatnocross\begin{pmat}({|....})
0_{q\times r} & & & & \cr\-
I & -G & & & \cr
& I & -G & & \cr
& & \ddots & \ddots & \cr
& & & \ddots & -G \cr
& & & & I \cr
\end{pmat},
\end{equation}
a matrix of size $(\nu r-q(r-1))\times (\nu-q)r$, whereas for $j\geq r$ we define
\begin{equation}\label{eq:kbarr+}
\bar{K}_{j}=
\pmatnocross\begin{pmat}({....})
I & -G & & & \cr
& I & -G & & \cr
& & \ddots & \ddots & \cr
& & & \ddots & -G \cr
& & & & I \cr
\end{pmat},
\end{equation}
a matrix of size $(\nu-q)r\times(\nu-q)r$.
\medskip\\
In what follows, we need the matrices $\bar{\ck}^i(z)=\bar{K}_0(z)\cdots\bar{K}_i(z)$, where the matrices $\bar{K}_i(z)$ have been introduced in (\ref{eq:kbarj}), (\ref{eq:kbarr-1}), (\ref{eq:kbarr+}).
Then $\bar{\ck}^i(z)$ is of size $\nu r\times(\nu r-q(i+1))$ for $i\leq r-2$ and of size $\nu r\times (\nu-q)r$ for $i\geq r-1$. Note that $\bar{\ck}^i$ is always of full column rank. The next lemma extends Equation~(\ref{eq:nkj}), obtained for the case $\nu\leq q$.

\begin{lemma}\label{lemma:fundamental} Let $0\leq i\leq r-1$. For $0\leq i<k$ one has
\begin{align}
\cn^k(z)\bar{\ck}^i(z)
& = \cn_{i+1}^{k-i-1}\begin{pmatrix} I_q\otimes I_r & 0 \\ 0 & I_{\nu-q}\otimes \cg_{i+1} \end{pmatrix}\nonumber\\
& = \big(\mathcal{R}^1_{ik}(z),\mathcal{R}^2_{ik}(z)\big)\label{eq:nki}
\end{align}
where $\mathcal{R}^1_{ik}(z)\in\rr^{q\times(r-i-1)q}$ and $\mathcal{R}^2_{ik}(z) \in\rr^{q\times(r-i-1)(\nu-q)}$ are explicitly given by
\begin{align*}
\mathcal{R}^1_{ik}(z) & = \frac{1}{(k-i-1)!}\big(\frac{{M}_{i+1}^{(k-i-1)}(z)}{0!},\cdots,\frac{{M}_{i+1}^{(k+q-i-2)}(z)}{(q-1)!}\big), \\
\mathcal{R}^2_{ik}(z) & = \frac{1}{(k-i-1)!}\big(\frac{M^{(k+q-i-1)}(z)G(z)^{i+1}}{q!},\cdots,\frac{M^{(k+\nu-i-2)}(z)G(z)^{i+1}}{(\nu-1)!}\big).
\end{align*}
For $i\geq k$ it holds that $\cn^k(z)\bar{\ck}^i(z)=0$.
\end{lemma}
\begin{proof}
The case $k=0$ has been verified in the proof of Proposition~\ref{prop:kerkz1}. Let therefore $k\geq 1$. To prove that the assertion holds true for  $i < k$, we assume the right hand side of formula~(\ref{eq:nki}) to be valid for $\cn^k(z)\bar{\ck}^{i-1}(z)$ and proceed by induction. To that end we multiply it by $\bar{K}_i$ and verify the answer. As before, we denote the $j$-th block column of $\bar{K}_i$ by $\bar{K}_i^j$, for $j=0,\ldots,\nu-1$. We will discern the four cases $j=0$, $j=1\ldots,q-1$, $j=q$ and $j=q+1,\ldots, \nu-1$.

Let $j=0$. Then the product $\cn^k\bar{\ck}^{i-1}\bar{K}_i^0$ becomes $\frac{1}{(k-i)!}M_i^{(k-i)}F_i$. The analogue of (\ref{eq:mff}), with $M_i$ and $F_i$ substituted for $M$ and $F$, yields that this equals $\frac{1}{(k-i-1)!}M_{i+1}^{(k-i-1)}$, as should be the case.

Let $1\leq j\leq q-1$. One obtains
\begin{equation}\label{eq:nki1}
\cn^k\bar{\ck}^{i-1}\bar{K}_i^j = \frac{1}{(k-i)!}\big(\frac{M_i^{(k+j-i-1)}}{(j-1)!}F_i'+\frac{M_i^{(k+j-i)}}{j!}F_i\big).
\end{equation}
The analogue of Equation~(\ref{eq:mff}) yields $M_i^{(k+j-i-1)}F_i'=-\frac{M_i^{(k+j-i)}}{k+j-i}$. Hence the right hand side of~(\ref{eq:nki1}) becomes
\[
\frac{1}{(k-i)!}\frac{1}{j!}\big(-\frac{j}{k+j-i}M_i^{(k+j-i)}F_i+M_i^{(k+j-i)}F_i\big) = \frac{1}{(k-i-1)!}\frac{1}{j!}\frac{M_i^{(k+j-i)}F_i}{k+j-i}.
\]
Invoking the analog of (\ref{eq:mff}) again, we can rewrite this as $\frac{1}{(k-i-1)!}\frac{M_{i+1}^{(k+j-i-1)}}{j!}$, a typical block of $,\mathcal{R}^1_{ik}$, as required.

Next we consider the more involved case $j=q$. In this case the block column $\bar{K}_i^q$ has entry $-\cg_iG$ on the $(q-1)$st row (see Lemma~\ref{lemma:ggg}) and $I$ on the $q$-th row. Hence we get
\begin{equation}\label{eq:nkij}
\cn^k\bar{\ck}^{i-1}\bar{K}_i^j = \frac{1}{(k-i)!}\big(-\frac{M_i^{(k-i+q-1)}}{(q-1)!}\cg_{i}G+\frac{M^{(k-i+q)}}{q!}G^i\big).
\end{equation}
Using (\ref{eq:mgkj}) we obtain $M_i^{(k-i+q-1)}\cg_iG=M^{(k-i+q-1)}G^{i+1}$.
In view of (\ref{eq:mgpq}), it holds that $M^{(k-i+q)}=(q+k-i)M^{(k-i+q-1)}G$. Hence we van rewrite the right hand side of (\ref {eq:nkij}) as
\[
\frac{1}{(k-i)!}\big(-\frac{M^{(k-i+q-1)}}{(q-1)!}+(q+k-i)\frac{M^{(k-i+q-1)}}{q!}\big)G^{i+1}
\]
which is equal to
\[
\frac{1}{(k-i-1)!}\frac{M^{(k-i+q-1)}}{q!}G^{i+1},
\]
the first block of $\mathcal{R}^2_{ik}$, as was to be shown.

Finally we treat the case $q+1\leq j\leq \nu-1$. The block columns $\bar{K}_i^j$ have $-G$ at the $(j-1)$st row and $I$ at the $j$th row. Hence, we obtain
\begin{equation}\label{eq:nkij2}
\cn^k\bar{\ck}^{i-1}\bar{K}_i^j = \frac{1}{(k-i)!}\big(-\frac{M^{(k-i+j-1)}G^i}{(j-1)!}G+\frac{M^{(k-i+j)}G^i}{j!}\big).
\end{equation}
Since $k-i+j-1>q$, we apply (\ref{eq:mgpq}) to get $M^{(k-i+j)}=(k-i+j)M^{(k-i+j-1)}G$, and the right hand side of (\ref{eq:nkij2}) reduces to
\[
\frac{1}{(k-i-1)!}\frac{M^{(k-i+j-1)}G^{i+1}}{j!},
\]
a typical block of $\mathcal{R}^2_{ik}$, as desired. This settles the proof of the validity of Equation~(\ref{eq:nki}).
\end{proof}

\begin{theorem}\label{thm:kernel2}
It holds that $\cn^k(z)\bar{\ck}^i(z)=0$, for $i\geq k$. For
$\mu\leq r-1$, the matrix $\bar{\ck}^{\mu-1}(z)$ is of size $\nu
r\times(\nu r-q\mu)$ and has full rank, equal to $\nu r - q\mu$. If
$\mu\geq r$, $\bar{\ck}^{\mu-1}(z)$ is of size $\nu r\times(\nu
-q)r$ and has full rank, equal to $(\nu -q)r$. Summarizing, the
kernel of $\cn(z)$ is $(\nu r- q\min\{\mu,r\})$-dimensional and
spanned by the columns of $\bar{\ck}^{\mu-1}(z)$. The rank of
$\cn(z)$ is equal to $q\min\{\mu,r\}<\nu r$ and therefore $\cn(z)$
never has full column rank.
\end{theorem}

\begin{proof}
We show that $\cn^k(z)\bar{\ck}^i(z)=0$ for $i\geq k$, for which it is clearly sufficient to show that $\cn^k(z)\bar{\ck}^k(z)=0$.
Starting point is Equation~(\ref{eq:nki}) for $i=k-1$. We have
\begin{align*}
\lefteqn{\cn^k(z)\bar{\ck}^{k-1}(z)=} \\
& \big(\frac{{M}_{k}^{(0)}(z)}{0!},\cdots,\frac{{M}_{k}^{(q-1)}(z)}{(q-1)!},\frac{M^{(q)}(z)G(z)^{k}}{q!},\cdots,\frac{M^{(\nu-1)}(z)G(z)^{k}}{(\nu-1)!}\big).
\end{align*}
We multiply this equation with the block columns $\bar{K}_k^j$ and, as above, we discern the case $j=0$, $1\leq j\leq q-1$, $j=q$ and $j=q+1,\ldots,\nu-1$.

For $j=0$ we get $\cn^k(z)\bar{\ck}^{k-1}\bar{K}_k^0=M_k^{(0)}F_k=uw_kF_k=0$,
whereas for $1\leq j\leq q-1$ one computes
\[
\cn^k\bar{\ck}^{k-1}\bar{K}_k^j= \frac{M_k^{(j-1)}}{(j-1)!}F_k'+\frac{M_k^{(j)}}{j!}F_k=0,
\]
in view of an analogue of (\ref{eq:mff}).

For $j=q$, we obtain
\begin{equation}\label{eq:gkk}
\cn^k\bar{\ck}^{k-1}\bar{K}_k^q= -\frac{M_k^{(q-1)}\cg_{k}G}{(q-1)!}+\frac{M^{(q)}G^k}{q!}.
\end{equation}
We can now use Equation~(\ref{eq:mqq-1}) and (\ref{eq:mgq}) to get
\[
\frac{M_k^{(q-1)}\cg_{k}G}{(q-1)!}=\frac{M^{(q-1)}G^{k+1}}{(q-1)!}=\frac{M^{(q)}G^{k}}{q!}.
\]
Hence, the right hand side of (\ref{eq:gkk}) is zero.

Next we consider the case $q+1\leq j \leq\nu+1$. We then get, parallel to (\ref{eq:nkij2}),
\[
\cn^k\bar{\ck}^{i-1}\bar{K}_k^j = \big(-\frac{M^{(j-1)}G}{(j-1)!}+\frac{M^{(j)}}{j!}\big)G^k,
\]
which is zero, in view of Equation~(\ref{eq:mqq-1}).

To show that $\cn(z)\bar{\ck}^{\mu-1}(z)=0$, one has to show that $\cn^k(z)\bar{\ck}^{\mu-1}(z)=0$, for all $k\leq \mu-1$, but this has implicitly been shown above.
The other statements in the theorem have already been addressed before. The theorem is proved.
\end{proof}

\subsection{The case $\nu\geq q+r$}

We follow the approach leading to Proposition~\ref{prop:kerkz2}. We observe that the matrix $\cn(z)$ for $\nu \geq q+r$ can be decomposed as
\[
\cn(z)=\begin{pmatrix}
\cn_*(z) & 0_{\mu q\times r(\nu-q-r-1)}
\end{pmatrix},
\]
where $\cn_*(z)$ is the ``$\cn(z)$ matrix" for the case $\nu=q+r-1$, since all derivatives of $M(z)$ of order higher than $q+r-2$ vanish. Let $\bar{\ck}_*^{\mu-1}(z)$ be the $\bar{\ck}^{\mu-1}(z)$ matrix for the case $\nu=q+r-1$.
Put
\[
\bar{\bar{\ck}}^{\mu-1}=\begin{pmatrix}
\bar{\ck}_*^{\mu-1} & 0 \\
0 & I
\end{pmatrix},
\]
where $I$ is the identity matrix of order $r(\nu-q-r-1)$.
If
$\mu<r$, then $\bar{\ck}_*^{\mu-1}(z)$ is of size $((q+r-1)r\times
(r(r-1)+(r-\mu)q)$, and if $\mu\geq r$, then it has size
$(q+r-1)r\times r(r-1)$.  Then $\bar{\bar{\ck}}^{\mu-1}(z)$ has size $\nu
r\times (\nu r-\mu q)$ for $\mu<r$ and has size $\nu r\times
(\nu-q)r$ for $\mu\geq r$. In short, $\bar{\bar{\ck}}^{\mu-1}(z)$ has dimensions $\nu r\times (\nu r - q\min\{\mu,r\})$.
\begin{theorem}\label{thm:kernel3}
Let $\nu\geq q+r$. The kernel of the matrix $\cn(z)$ is spanned by
the columns of $\bar{\bar{\ck}}^{\mu-1}(z)$, has dimension  $\nu
r-\mu q$ if $\mu<r$ and dimension $(\nu-q)r$ if $\mu\geq r$. So $\dim\ker(\cn(z))=\nu r-q\min\{\mu,r\}$.
\end{theorem}

\begin{proof}
Similar to the proof of Proposition~\ref{prop:kerkz2}, using the results of Theorem~\ref{thm:kernel2} for the case $\nu=q+r-1$.
\end{proof}

\section{Intermezzo, properties of $\ca^0(z)$}\label{section:intermezzo}
\setcounter{equation}{0}

The results of this section will be used in Section~\ref{section:rightinverse}, where we want to find (special) right inverses of the matrix $\cm^0(z)$.

We focus on the matrix
$\mathcal{A}^0=(A^0,\ldots,A^{\nu-1})\in \rr^{q\times \nu r}$, the first block row of $\ca$, the matrix defined in Section~\ref{section:ca}. One
directly sees that the rank of $\mathcal{A}^0$ is equal to
$\min\{q,\nu\}$, although it also follows from Theorem~\ref{thm:ao1} with $\mu=1$. Hence $\mathcal{A}^0$ has full rank iff $\nu\geq
q$. We introduce the matrix $\mathcal{B}^0\in\rr^{\nu r\times q}$
consisting of the $r\times q$ blocks $B^k$ as follows.
\begin{equation}\label{eq:b}
\mathcal{B}^0=\begin{pmatrix} B^0 \\ \vdots \\ B^{\nu-1}
\end{pmatrix}
\end{equation}
where each $B^k$ has elements
\begin{equation*}
B^k_{ij}= \left\{ \begin{array}{ll} (-1)^{i}{q \choose k+1} & \mbox{ if } i+j=k \\
0 & \mbox{ else, }
\end{array}
\right.
\end{equation*}
for $i=0,\ldots,r-1$, $j=0,\ldots,q-1$.

\begin{lemma}\label{lemma:a0b}
Let $\nu\geq q$. Then $\mathcal{A}^0$ has full row rank and
$\mathcal{A}^0\mathcal{B}^0=I$. In other words, $\mathcal{B}^0$ is a
right inverse of $\mathcal{A}^0$.
\end{lemma}
\begin{proof}
We have to compute the $ij$-elements of
$T:=\sum_{k=0}^{\nu-1}A^kB^k$. Using the definitions of the
matrices $A^k$ and $B^k$ that only have nonzero entries on
corresponding anti-diagonals, we see that $A^kB^k$ is a diagonal
matrix. Hence we only have to consider the $ii$-entries of $T$.
Note that $B^k=0$ for $k\geq q$. One obtains
\begin{align*}
T_{ii} & = \sum_{k=0}^{q-1} (A^kB^k)_{ii} \\
& = \sum_{k=0}^{q-1}{k \choose i}(-1)^{k-i}{q \choose k+1}
\\
& = \frac{q!}{i!(q-1-i)!}\sum_{k=i}^{q-1}{q-1-i \choose
k-i}\frac{(-1)^{k-i}}{k+1}.
\end{align*}
To compute the latter summation, we write it as
\begin{align*}
\int_0^1\sum_{k=i}^{q-1}{q-1-i \choose k-i}(-1)^{k-i}x^k\,\dd x &
= \int_0^1\sum_{j=0}^{q-1-i}{q-1-i \choose
j}(-x)^jx^i\,\dd x \\
& = \int_0^1 (1-x)^{q-1-i}x^i\,\dd x \\
& = B(q-i,i+1),
\end{align*}
by definition of the $\beta$-function $B(\cdot,\cdot)$. Using the
well-known fact that this can be computed in terms of
$\Gamma$-functions
($B(\alpha,\beta)=\Gamma(\alpha)\Gamma(\beta)/\Gamma(\alpha+\beta)$)
we obtain
\[
B(q-i,i+1)=\frac{(q-1-i)!i!}{q!}.
\]
It follows that $T_{ii}=1$.
\end{proof}
We need some additional properties.
\begin{lemma}\label{lemma:bj}
It holds that $X^k:=B^kS_q^\top +S_r B^k=0$, if $k\geq q$ or $k\leq
r-1$. For the case $r\leq k \leq q-1$ (which requires $q>r$) only
the last row of this matrix is nonzero. In fact, this row is equal
to $(-1)^{r-1}{q\choose k+1}e^\top_{k-r}$, with the convention
that $e_i$ denotes the $i$-th basis vector of $\rr^q$
($i=0,\ldots,q-1$).
\end{lemma}
\begin{proof}
We compute the $ij$-element of $X^k=B^kS_q^\top +S_r B^k$. For
$i=0,\ldots,r-1$ and $j=0,\ldots,q-1$ it is equal to
\begin{align*}
X^k_{ij} & = \sum_{l=0}^{q-1}B^k_{il}1_{\{l=j+1\}}+\sum_{l=0}^{r-1}1_{\{i+1=l\}}B^k_{lj} \\
& = \sum_{l=0}^{q-1}B^k_{i,j+1}1_{\{l=j+1\}}+\sum_{l=0}^{r-1}1_{\{i+1=l\}}B^k_{i+1,j} \\
& = B^k_{i,j+1}1_{\{0\leq j+1\leq q-1\}}+1_{\{0\leq i+1\leq r-1\}}B^k_{i+1,j} \\
& = B^k_{i,j+1}1_{\{0\leq j\leq q-2\}}+1_{\{0\leq i\leq r-2\}}B^k_{i+1,j} \\
& = {q \choose k+1}1_{\{i+j+1=k\}}\big((-1)^i1_{\{0\leq j\leq q-2\}}+(-1)^{i+1}1_{\{0\leq i\leq r-2\}}\big) \\
& = {q \choose k+1}1_{\{i+j+1=k\}}(-1)^i\big(1_{\{0\leq j\leq
q-2\}}-1_{\{0\leq i\leq r-2\}}\big).
\end{align*}
Clearly, for $i=0,\ldots,r-2$ and $j=0,\ldots,q-2$, the last
expression in the display equals zero, as is the case for $i=r-1$
and $j-q-1$. We next consider the two remaining cases, the first
being $i\leq r-2$ and $j=q-1$. Since we only have to consider
$i=k-j-1$, we get $i=k-q$, which has to be nonnegative, so $k\geq
q$. But then the binomial coefficient is equal to zero. The
remaining case is $i=r-1$. Then we only have to consider $j=k-r$,
the other values of $j$ again give zero. Note that this implies
that $k\geq r$ is necessary to get a nonzero outcome, whereas we
already know that also $k\leq q-1$ is necessary. Hence nonzero
elements in the last row of $X^k$ can only occur if $r\leq q-1$.
Under this last condition we find $X^k_{r-1,j}={q\choose
k+1}(-1)^{r-1}1_{\{j=k-r\}}$. Hence the bottom row of $X^k$ equals
${q\choose k+1}(-1)^{r-1}(1_{\{k=r\}},\ldots,1_{\{k=r+q-1\}})$,
which is equal to ${q\choose k+1}(-1)^{r-1}e^\top_{k-r}$, for
$k=r,\ldots,q-1$.
\end{proof}

\begin{remark}\label{remark:bjnot}
Here is an example where $X^k$ as defined in Lemma~\ref{lemma:bj} is not equal to zero. Take $k=r=1$ and $q=2$. Then $B^1=(0 \quad 1)$ and $X^1=(1\quad 0)$.
\end{remark}

\begin{proposition}\label{prop:gconstant}
Define $H_k:\rr\to\rr^{r\times q}$ by $H_k(z)=\widetilde{W}_r(z)
D_r^{-1}B^kD_q^{-1}\widetilde{U}_q(z)$. Then $H_k$ is a constant
mapping, $H_k(z)\equiv D_r^{-1}B^kD_q^{-1}$, under the condition
$k\geq q$ or $k\leq r-1$.
\end{proposition}

\begin{proof}
First we prove the following auxiliary results.
One has
\begin{align}
\widetilde{U}_q'(z) & =\widetilde{U}_q(z)D_qS_q^\top D_q^{-1} \label{eq:uprime}\\
\widetilde{W}_r'(z) & =D_r^{-1}S_r D_r\widetilde{W}_r(z).\label{eq:wprime}
\end{align}
Equation~(\ref{eq:uprime}) follows from the definition of $\widetilde{U}_q(z)$ and the elementary
identity $U_q'(z)=U_q(z)S_q^\top $.  Equation~(\ref{eq:wprime}) can be proved similarly.

We now compute
\begin{align*}
H_k'(z) & = \widetilde{W}_r'(z) D_r^{-1}B^kD_q^{-1}\widetilde{U}_q(z) +
\widetilde{W}_r(z)
D_r^{-1}B^kD_q^{-1}\widetilde{U}_q'(z) \\
& = D_r^{-1}S_r D_r \widetilde{W}_r(z)
D_r^{-1}B^kD_q^{-1}\widetilde{U}_q(z) + \widetilde{W}_r(z)
D_r^{-1}B^kD_q^{-1}\widetilde{U}_q(z) D_qS_q^\top D_q^{-1},
\end{align*}
according to Equations~(\ref{eq:uprime}) and (\ref{eq:wprime}). Putting
$\hat{S}_r=D_r^{-1}S_rD_r$ and $\hat{S}_q=D_qS_q D_q^{-1}$, we see that $H_k$
satisfies the linear differential equation
\begin{equation}\label{eq:lde}
H_k'(z)=\hat{S}_r H_k(z) + H_k(z) \hat{S}_q^\top.
\end{equation}
This equation has a unique  solution and we claim that it is given
by the constant function as asserted. To that end we check
\begin{align*}
\hat{S}_r D_r^{-1}B^kD_q^{-1} + D_r^{-1}B^kD_q^{-1}\hat{S}_q^\top & =
D_r^{-1}(S_r B^k+B^kS_q^\top)D_q^{-1} \\
& = 0,
\end{align*}
by Lemma~\ref{lemma:bj}, since $k\geq q$ or $k\leq r-1$.
Furthermore, we have $H_k(0)=D_r^{-1}B^kD_q^{-1}$, since
$\widetilde{U}_q(0)=I_q$.
\end{proof}
\begin{remark}
Equation~(\ref{eq:lde}) has as the general solution
\begin{equation}\label{eq:hgh}
H_k(z)=\exp(\hat{S}_r z)H_k(0)\exp(\hat{S}_q^\top z),
\end{equation}
where the exponentials can be computed as finite sums, since $S_q^\top $
and $S_r$ are nilpotent. Elementary computations yield for
instance that the $ij$-element of $\exp(\hat{S}_q^\top z)$ is given by
${i\choose j}z^{i-j}$ for $i\geq j$ and zero otherwise. Hence we
obtain $\exp(\hat{S}_qz)=\widetilde{U}_q(z)$, which is in agreement with the
definition of $H_k(z)$.

An example of a solution that is not constant is obtained for
$r=1$ and $q=2$. For the  case $k=1$ one finds directly from the definition of $H_k(z)$ that $H_1(0)=B^1=(0\quad 1)$ and
$H_k(z)=\begin{pmatrix}z & 1\end{pmatrix}$ in view of Remark~\ref{remark:bjnot}. This is in agreement with Equation~(\ref{eq:hgh}), whose right hand side is equal to
\[
(0\quad 1)\begin{pmatrix}
1 & 0 \\
z & 1
\end{pmatrix}.
\]
\end{remark}

\section{The equation $\cm^0(z)C=I$}\label{section:rightinverse}
\setcounter{equation}{0}

We return to one of our original
aims, finding a right inverse of the $q\times \nu r$ matrix
$\cm^0(z)=(M(z),\ldots,M^{(\nu-1)}(z))\in\rr^{q\times r\nu}$. Recall from Theorem~\ref{thm:ao1} and Proposition~\ref{prop:factorM} that $\cm^0(z)$ has rank equal to $\min\{\nu,q\}$. Hence the matrix is
of full rank iff $\nu\geq q$. Equations like $\cm^0(z)X=b$ will in
general not have a solution $X$ for a given $b\in \rr^{\nu r\times
1}$, if $\nu<q$. In fact, we are interested in solutions $X$ that
are independent of $z$. It is easy to see that such solutions only exist if $b=0$ and then $X=0$. The uninteresting case $\nu < q$ will
therefore be ignored and the standing assumption in the remainder of this section is $\nu\geq q$. Under this assumption, there are two subcases to discern, $r\geq q$ and $r<q$.

\begin{proposition}\label{prop:mki}
Assume that $r\geq q$ and $\nu\geq q$. Let $I_q$ be the
$q$-dimensional unit matrix. There exists a constant (not
depending on $z$) matrix $C\in\rr^{\nu r\times q}$ such that
$\mathcal{M}^0(z)C=I_q$ for all $z$. The equation $\cm^0(z)X=b$
for $b\in\rr^{q}$ then has the constant solution $X=Cb$. The constant matrix $C$ is unique iff $\nu=q$. In all cases one can take $C=(I_\nu\otimes
D_r^{-1})\mathcal{B}^0D_q^{-1}$, with $\mathcal{B}^0$ as in (\ref{eq:b}).
\end{proposition}


\begin{proof}
In this proof we simply write $I$ for $I_q$. Suppose that we have
found a constant matrix $C$ with the desired property
\begin{equation}\label{eq:mk0}
\mathcal{M}^0(z)C=I.
\end{equation}
By differentiation of (\ref{eq:mk0}) $k$ times, with $k=0,\ldots,r-1$,
we obtain, recall the definition of $\mathcal{M}(z)$ with $\mu=r$, that
\begin{equation}\label{eq:mk}
\mathcal{M}(z)C=
\begin{pmatrix}
I \\
0 \\
\vdots \\
0
\end{pmatrix}.
\end{equation}
We note that now $\cm(z)$ is of size $rq\times r\nu$ and has rank equal to $rq$. Hence $\cm(z)$ has a right inverse, $\cm(z)^+$ say, and a true inverse in the case that $\nu=q$, see e.g.\ Corollary~\ref{cor:dets}.
It follows  that $C$ should be the first block-column of
$\mathcal{M}(z)^{+}$.  Next we use the factorization (\ref{eq:m}) and note that also $\ca$ has a right inverse, $\ca^+$ say. Then
\[
\cm(z)^+=(I_\nu\otimes W_r(z)^{-1})\ca^+(I_r\otimes U_q(z)^{-1}).
\]
Hence, we can choose
\[
C=(I_\nu\otimes W_r(z)^{-1})\ca^+(I_r\otimes U_q(z)^{-1})\begin{pmatrix}
I \\
0 \\
\vdots \\
0
\end{pmatrix},
\]
which means that $C$ is the first block-column of $\cm(z)^+$, so
\begin{equation}\label{eq:c}
C=(I_\nu\otimes W_r(z)^{-1})\ca^+\begin{pmatrix}
U_q(z)^{-1} \\
0 \\
\vdots \\
0
\end{pmatrix} = (I_\nu\otimes W_r(z)^{-1})(\ca^0)^+U_q(z)^{-1},
\end{equation}
where $(\ca^0)^+$ is a right inverse of the matrix $\ca^0$, since $\ca^0$ is the first block-row of $\ca$.
But, a right inverse of $\ca^0$ is in Proposition~\ref{lemma:a0b} shown to be $\mathcal{B}^0$. Therefore,
we can now
explicitly pose our candidate for $C$,
\begin{equation}\label{eq:kz}
C=(I\otimes W_r(z)^{-1})\mathcal{B}^0U_q(z)^{-1},
\end{equation}
where $\mathcal{B}^0$ as defined
in~(\ref{eq:b}). Hence we have to show that \\
(1) the matrix $C$ in~(\ref{eq:kz}),  in fact doesn't
depend on $z$,
\\
(2) $\mathcal{M}^0(z)C=I$.
Using
 matrices introduced in Section~\ref{section:factorization}, we write
\begin{align*}
C & = (I\otimes
(\widetilde{W}_r(z)^{-1}D_r^{-1}))\mathcal{B}^0D_q^{-1}\widetilde{U}_q(z)^{-1} \\
& = (I\otimes
\widetilde{W}_r(-z)D_r^{-1})\mathcal{B}^0D_q^{-1}\widetilde{U}_q(-z).
\end{align*}
Decomposing $C$ as
\[
C=\begin{pmatrix} C^0 \\ \vdots \\ C^{\nu-1}\end{pmatrix},
\]
where each block $C^k$ ($k=0,\ldots,\nu-1$) is of size $r\times q$, we get
\[
C^k=\widetilde{W}_r(-z)^\top D_r^{-1} B^kD_q^{-1}\widetilde{U}_q(-z).
\]
Hence we see that $C^k=H_k(-z)$, which was in
Proposition~\ref{prop:gconstant} shown to be constant and equal to
$D_r^{-1} B^kD_q^{-1}$, if we have $k\leq r-1$ or $k\geq q$. Obviously, this is true of $k=0,\dots,r-1$, but for $k=r,\ldots,\nu-1$, we have $k\geq r\geq q$ by assumption.
This
proves claim (1). Since $C$ is constant in $z$, we can take $z=0$ in (\ref{eq:kz}).

For the second one we have
\begin{align*}
\mathcal{M}^0(z)C & = U_q(z)\mathcal{A}^0(I\otimes W_r(z))C \\
& = U_q(z)\mathcal{A}^0(I\otimes W_r(z))(I\otimes
W_r(z)^{-1})\mathcal{B}U_q(z)^{-1} \\
& = U_q(z)\mathcal{A}^0 \mathcal{B}U_q(z)^{-1} \\
& = I,
\end{align*}
in view of Lemma~\ref{lemma:a0b}. Finally, if $\nu=q$, then $\cm(z)$ is invertible, which implies that $C$ is the unique constant matrix solving $\cm^0(z)C=I$, since in this case Equation~(\ref{eq:mk}) has a unique solution.
\end{proof}
\begin{remark}
The special choice $(\bar{\ca}^0)^+=\cb^0$ in the proof of Proposition~\ref{prop:mki} is rather crucial in finding a right inverse of $\cm^0(z)$ that doesn't depend on $z$. We illustrate this with the following example. Our point of departure is Equation~(\ref{eq:c}) with $\mu=1$. 

Recalling (\ref{eq:aa0}), we can take $(\ca^0)^+=\cl_{\nu,r}(0)^{-1}\bar{\ca}^+\cl_{1,q}(0)^{-1}$, with $\bar{\ca}^+$ any right inverse of $\bar{\ca}$. We choose $\bar{\ca}^+=\bar{\ca}^\top$ and compute 
\[
\bar{\cb}:=\cl_{\nu,r}(0)^{-1}\bar{\ca}^\top\cl_{1,q}(0)^{-1}=\cl_{\nu,r}(0)^{-1}\bar{\ca}^\top\in\rr^{\nu r\times q}, 
\]
because $\cl_{1,q}(0)=I_q$. Since $\bar{\ca}^\top_j=f_0e_j^\top$ (Theorem~\ref{thm:ao1}) and $\cl_{\nu,r}(0)^{-1}_{ij}={i\choose j}(-S_r)^{i-j}$, for the $k$-th block $\bar{B}^k$ of $\bar{\cb}$ we get $\bar{B}^k=\sum_{l=0}^{\nu-1}{l\choose k}(-1)^{l-k}f_{k-l}e_l^\top=f_0e_k^\top$ ($k=0,\ldots,\nu-1)$. We conclude that $\bar{\cb}=\bar{\ca}^\top$.

In order to see that this may result in a right inverse of $\cm^0(z)$ that depends on $z$, we choose $q=r=\nu=2$. The conditions of Proposition~\ref{prop:mki} are then satisfied. We have
\[
\cm^0(z)=
\begin{pmatrix}
1 & z & 0 & 1 \\
z & z^2 & 1 & 2z
\end{pmatrix}
\]
and it follows from the above that
\[
\cm^0(z)^+=(I_2\otimes W_2(z)^{-1})\bar{\cb}U_2(z) =
\begin{pmatrix}
1 & 0 \\
0 & 0 \\
-z & 1 \\
0 & 0
\end{pmatrix}.
\]
We close this remark by noting that there also other right inverses of $\cm^0(z)$, depending on $z$, but still having a simple structure. An example (essentially taken from~\cite{ks2007}, where it was only given for $\nu=q=r$ in  a slightly different situation) is
\[
\cm^0(z)^+=
\begin{pmatrix}
U_q(z)^{-1} \\
0_{(\nu-q)\times q}
\end{pmatrix}
\otimes
f_0,
\]
where $f_0$ is the first basis vector of $\rr^r$. This follows from the easy to verify identity
\[
\cm^0(z)(I_\nu\otimes f_0)=\begin{pmatrix} U_q(z) & 0_{q\times(\nu-q)} \end{pmatrix}.
\]
\end{remark}
The assertion of Proposition~\ref{prop:mki} is not true if $r<q$ (the second subcase). Indeed, in the proof of this proposition we used the fact that all $C^k$ are indeed constant matrices, under the condition $r\geq q$. If this is not the case, $r<q$, the matrices $C^k$ for $k=r, \ldots,q-1$ are not constant, in view of Proposition~\ref{prop:gconstant}. Let us give an example to illustrate this.
Consider the case $q=\nu=2$ and $r=1$. Then
\[
\cm^0(z)=\begin{pmatrix}
1 & 0 \\ z & 1
\end{pmatrix}
\]
and the equation $\cm^0(z)C=I_2$ has the {\em unique} on $z$ depending solution
$C=C(z)=\cm^0(z)^{-1}$.
\medskip\\
We now treat the case $\nu> q$ in more detail. To that end we need the following auxiliary result.
\begin{lemma}\label{lemma:kerm0m}
The subspace of the kernel of $\cm^0(z)$ that consists of vectors that are constant in $z$, i.e.\ the intersection $\bigcap_z\ker(\cm^0(z))$, is $(\nu-q)^+r$-dimensional. This subspace is equal to the kernel of $\cm(z)$ with $\mu=r$, which is the same for all $z$ and hence can be parametrized free of $z$.
\end{lemma}
\begin{proof}
The first observation is that a vector $x$ in $\ker \cm^0(z)$ that doesn't depend on $z$ also satisfies $\cm(z)x=0$ for arbitrary $\mu$, in particular for $\mu=r$. The case $\nu\leq q$ follows from Theorem~\ref{thm:kernel1}, since in this case the kernel of $\cm(z)$ is the null space for all $\mu\geq r$.

Let then $\nu>q$. Let $x$ be a column vector consisting of $r$-dimensional sub-vectors $x_0,\ldots,x_{\nu-1}$ that don't depend on $z$. Recalling that $\cm^0$ consists of a row of blocks $(uw)^{(n)}$, we have
\begin{align*}
\cm^0x & =\sum_{n=0}^{\nu-1}(uw)^{(n)}x_n \\
& = \sum_{n=0}^{\nu-1}\sum_{k=0}^n{n\choose k}u^{(k)}w^{(n-k)} \\
& = \sum_{k=0}^{\nu-1}u^{(k)}\sum_{n=k}^{\nu-1}{n\choose k}w^{(n-k)}x_n.
\end{align*}
The vectors $u^{(k)}$ are zero for $k\geq q$ and otherwise linear independent. Hence, to have the above sum equal to zero is equivalent to
\[
\sum_{n=k}^{\nu-1}{n\choose k}w^{(n-k)}x_n=0 \mbox{ for }k=0,\ldots,q-1.
\]
The equation for arbitrary $1\leq k\leq q-1$ can be differentiated to get
\[
\sum_{n=k}^{\nu-1}{n\choose k}w^{(n+1-k)}x_n=0,
\]
which, subtracting from the equation for $k-1$ yields
\[
{n\choose k-1}w^{(0)}x_{k-1}=0,
\]
valid for $k=1,\ldots, q$. The only constant solutions to these equations are the zero vectors, so $x_0=\cdots =x_{q-2}=0$.

On the other hand,  for $k=q-1$ we keep the equation
\[
\sum_{m=0}^{\nu-q}{m+q-1\choose q-1}w^{(m)}x_{m+q-1}=0.
\]
Now we relabel the unknowns by setting $y_m={m+q-1\choose q-1}x_{m+q-1}$ to get
\begin{equation*}
\sum_{m=0}^{\nu-q}w^{(m)}y_m=0.
\end{equation*}
We differentiate this equation $j$ times, with $j=0,\ldots,r-1$ to get
\begin{equation}\label{eq:wy}
\begin{pmatrix}
w^{(0)} & \cdots & w^{(\nu-q)} \\
\vdots & & \vdots \\
w^{(r-1)} & \cdots & w^{(\nu-q+r-1)}
\end{pmatrix}
y=0,
\end{equation}
where $y$ is obtained by stacking the $y_m$. The first block-column in the above matrix is $W$, the second can be written as $S_rW$, up to the last one equal to $S_r^{\nu-q}W$. Hence the above system of equations can be compactly written as
\[
\begin{pmatrix}
W & S_rW & \cdots & S_r^{\nu-q}W
\end{pmatrix}
y=0.
\]
Let $\Delta$ be the diagonal matrix with elements $\Delta_{ii}=i$. A simple computation shows that $S_rW=WS_r\Delta$, and therefore $S_r^kW=W(S_r\Delta)^k$. Writing $S_r\Delta=T$, we can rewrite the last equation in $y$ as
\[
\begin{pmatrix}
W & WT & \cdots & WT^{\nu-q}
\end{pmatrix}
y=0.
\]
Since $W=W_r(z)$ is invertible for any $z$, this reduces to
\[
\begin{pmatrix}
I & T & \cdots & T^{\nu-q}
\end{pmatrix}
y=0.
\]
Since the coefficient matrix has full row rank equal to $r$, its kernel has dimension $(\nu-q+1)r-r=(\nu-q)r$. Actually, this kernel is spanned by the columns of the $(\nu-q+1)r\times(\nu-q)r$ matrix
\[
\begin{pmatrix}
-T & & & \\
I_r & -T & & \\
& I_r & & \\
& & \ddots & -T \\
& & & I_r
\end{pmatrix}.
\]
This proves the claim.
\end{proof}
\begin{remark}
The result of Lemma~\ref{lemma:kerm0m} is also valid for $\mu>r$. This can be seen from Equation~(\ref{eq:wy}). Indeed, if $\mu>r$ one has to extend the coefficient matrix with additional rows, all involving derivatives $w^(k)$, with $k\geq r$. But these are all equal to zero.
\end{remark}
One might think that the assertion of the lemma can alternatively be proven by explicitly computing the matrix $\bar{\ck}^{\mu-1}(z)$ for $\mu=r$ (noting that $(D_\nu^{-1} \otimes I_r)\bar{\ck}^{\mu-1}(z)$ represents the kernel of $\cm(z)$ in view of (\ref{eq:mdnd})) and showing that it is not depending on $z$. It turns out that this idea is false, as shown by the following simple example.

Let $q=1$, $r=\mu=\nu=3$. We compute $\bar{\ck}^{\mu-1}(z)$ and show that it is not free of $z$. According to the results of Section~\ref{section:2ndcase} we find for  $\bar{K}_0$, $\bar{K}_1$ and $\bar{K}_2$ the following.
\[
\bar{K}_0(z)=
\pmatnocross\begin{pmat}({.|..|...})
-z & 0 & 0 & -1 & -z & 0 & 0 & 0 \cr
1 & -z & 0 & 0 & -1 & 0 & 0 & 0 \cr
0 & 1 & 0 & 0 & 0 & 0 & 0 & 0  \cr\-
0 & 0 & 1 & 0 & 0 & 0 & -1 & -z \cr
0 & 0 & 0 & 1 & 0 & 0 & 0 & -1  \cr
0 & 0 & 0 & 0 & 1 & 0 & 0 & 0 \cr\-
0 & 0 & 0 & 0 & 0 & 1 & 0 & 0 \cr
0 & 0 & 0 & 0 & 0 & 0 & 1 & 0 \cr
0 & 0 & 0 & 0 & 0 & 0 & 0 & 1 \cr
\end{pmat},
\]
\[
\bar{K}_1(z)=
\pmatnocross\begin{pmat}({|..|...})
-z & 0 & 0 & -1 & 0 & 0 & 0 \cr
1 & 0 & 0 & 0  & 0 & 0 & 0 \cr\-
0 & 1 & 0 & 0 & 0 & -1 & -z \cr
0 & 0 & 1 & 0 & 0 & 0 & -1  \cr
0 & 0 & 0 & 1 & 0 & 0 & 0 \cr\-
0 & 0 & 0 & 0 & 1 & 0 & 0 \cr
0 & 0 & 0 & 0 & 0 & 1 & 0 \cr
0 & 0 & 0 & 0 & 0 & 0 & 1 \cr
\end{pmat},\quad
\bar{K}_2(z)=
\pmatnocross\begin{pmat}({..|...})
0 & 0 & 0 & 0 & 0 & 0 \cr\-
1 & 0 & 0 & 0 & -1 & -z \cr
0 & 1 & 0 & 0 & 0 & -1  \cr
0 & 0 & 1 & 0 & 0 & 0 \cr\-
0 & 0 & 0 & 1 & 0 & 0 \cr
0 & 0 & 0 & 0 & 1 & 0 \cr
0 & 0 & 0 & 0 & 0 & 1 \cr
\end{pmat},
\]
and the product $\bar{\ck}^2(z)=\bar{K}_0(z)\bar{K}_1(z)\bar{K}_2(z)$ yields the kernel of $\cm(z)$ spanned by the columns of
\[
(D_3^{-1} \otimes I_3)\bar{\ck}^2(z)=
\pmatnocross\begin{pmat}({..|...})
0 & -1 & 0 & 0 & 0 & 2 \cr
0 & 0 & -2 & 0 & 0 & 0 \cr
0 & 0 & 0 & 0 & 0 & 0 \cr\-
1 & 0 & 0 & 0 & -3 & -3z \cr
0 & 1 & 0 & 0 & 0 & -3  \cr
0 & 0 & 1 & 0 & 0 & 0 \cr\-
0 & 0 & 0 & \half & 0 & 0 \cr
0 & 0 & 0 & 0 & \half & 0 \cr
0 & 0 & 0 & 0 & 0 & \half \cr
\end{pmat}.
\]
We see that the last column of $\bar{\ck}^2(z)$ has a term $-3z$ in the fourth row, hence this parametrization of $\ker(\cm(z))$ is not the one we are looking for. The reparametrization of $\ker(\cm(z))$ given by $(D_3^{-1} \otimes I_3)\bar{\ck}^2(z)R(z)$ with
\[
R(z)=
\pmatnocross\begin{pmat}({..|...})
1 & 0 & 0 & 0 & 5 & 6z \cr
0 & 1 & 0 & 0 & 0 & 4 \cr
0 & 0 & 1 & 0 & 0 & 0 \cr\-
0 & 0 & 0 & 2 & 0 & 0 \cr
0 & 0 & 0 & 0 & 2 & 0 \cr
0 & 0 & 0 & 0 & 0 & 2 \cr
\end{pmat}
\]
yields $(D_3^{-1} \otimes I_3)\bar{\ck}^2(z)R(z)=:\hat{\ck}$, with
\[
\hat{\ck}=
\pmatnocross\begin{pmat}({..|...})
0 & -1 & 0 & 0 & 0 & 0 \cr
0 & 0 & -2 & 0 & 0 & 0 \cr
0 & 0 & 0 & 0 & 0 & 0 \cr\-
1 & 0 & 0 & 0 & -1 & 0 \cr
0 & 1 & 0 & 0 & 0 & -2  \cr
0 & 0 & 1 & 0 & 0 & 0 \cr\-
0 & 0 & 0 & 1 & 0 & 0 \cr
0 & 0 & 0 & 0 & 1 & 0 \cr
0 & 0 & 0 & 0 & 0 & 1 \cr
\end{pmat},
\]
which is free of $z$. The procedure in the proof of Lemma~\ref{lemma:kerm0m} (in this case $y_m=x_m$, since $q=1$) yields a kernel of $\cm(z)$ spanned by column vectors that are not depending on $z$. The result of that procedure is the matrix $\hat{\ck}$ above.

\begin{proposition}\label{prop:mki2}
Let $q<\nu\leq r$. Then the equation $\cm^0(z)C=I_q$ admits a
constant solution $C\in\rr^{\nu r\times q}$. The dimension of the
affine space of constant solutions is equal to $(\nu-q)rq$.
\end{proposition}

\begin{proof}
According to Proposition~\ref{prop:mki} a constant solution $C$ exists.
Any other constant right inverse $C'$ is such that the $q$ columns of $C'-C$ belongs to the kernel of $\cm^0$ and hence to the kernel of $\cm$ for $\mu=r$. In view of Lemma~\ref{lemma:kerm0m}, a basis of this kernel can be obtained can be obtained by choosing $(\nu-q)r$ linearly independent vectors. Applying this result to each the columns of $C'-C$, we obtain the result.
\end{proof}
To illustrate the fact that the constant solution $C$ of
Proposition~\ref{prop:mki2} is in general not unique, we consider
the case $q=1$, $r=\nu=2$. Then $\cm^0(z)=\begin{pmatrix} 1 & z & 0
& 1\end{pmatrix}$ and all constant solutions are given by
$C=C_{a,b}=\begin{pmatrix} a & 0 & b & 1-a\end{pmatrix}^\top$ with
$a,b \in\rr$, which form an affine space of dimension $(\nu-q)qr=2$.
\medskip\\
Of course all right inverses of $\cn^0(z)$, also those that depend on $z$, are given by a much larger affine subspace.
Assume $\nu\geq q$ and let $C_0(z)$ be any right inverse of $\cn^0(z)$.
Then any matrix $C(z)=C_0(z)+X(z)$, with
$X(z)\in \rr^{\nu r\times q}$ a matrix whose columns belong to $\ker\cn^0(z)$ is a right inverse. Since, $\dim\ker\cn^0(z)=\nu r -q$,
the affine subspace of these right inverses has dimension $(\nu r-q)q$.
\medskip\\
The natural extension of the equation $\cm^0(z)C=I_q$ is $\cm(z)C=I_{\mu q}$, with $\cm(z)$ of order $\mu q\times \nu r$ and $I_{\mu q}$ the identity matrix of order $\mu q$.
The matrix $\cm(z)$ has rank equal to $\min\{\mu,r\}\times\min\{\nu,q\}$ and therefore has full row rank if and only if $\nu\geq q$ and $\mu\leq r$. Hence, under the latter condition, and only then, a right inverse exists, and the equation $\cm(z)C=I_{\mu q}$  has a solution. This equation can be decomposed as
\[
\begin{pmatrix}
\cm^0(z) \\
\vdots \\
\cm^{\mu-1}(z)
\end{pmatrix}
\begin{pmatrix}
C_{0}& \cdots & C_{\mu-1}
\end{pmatrix}
=
\begin{pmatrix}
I_q & 0 & \cdots & 0 \\
0 & I_q & & \vdots \\
\vdots & & \ddots & \vdots \\
0 & \cdots & 0 & I_q
\end{pmatrix},
\]
where every $C_j$ has size $\nu r\times q$. Parallelling our
previous aim, also here one could be interested in finding solutions
$C$ that are constant in $z$. For $C_0$ we are in the previous
situation, since a constant $C_0$ satisfying $\cm^0(z)C_0=I_q$, also
satisfies $\cm^k(z)C_0=0$ for all $k\geq 1$. The situation for the
the other $C_k$ is different. Consider for example $C_1$. It should
satisfy $\cm^0(z)C_1=0$ and $\cm^1(z)C_1=I_q$. However, this is
impossible for a $C_1$ that is constant  in $z$, since
differentiating $\cm^0(z)C_1=0$ yields $\cm^1(z)C_1=0$. We conclude
that the equation $\cm(z)C=I_{\mu q}$ for $\mu\geq 2$ has no constant
solutions.

Nonconstant solutions are for instance Moore-Penrose
inverses. These can be obtained by using the Moore-Penrose inverse
of the matrix $\bar{\ca}$. It follows from the proof of
Theorem~\ref{thm:ao1} that for $\mu\leq r$ and $\nu\geq q$, the
matrix $\bar{\ca}(\bar{\ca})^\top$ is the identity matrix. Hence
$(\bar{\ca})^\top$ is a right inverse of $\bar{\ca}$. Using then
Theorem~\ref{thm:azfactor1} and Proposition~\ref{prop:factorM} we obtain
that
\[
\cm(z)^+=(I_\nu\otimes
W_r(z)^{-1})\cl_{\nu,r}^{-1}\bar{\ca}^\top\cl_{\mu,q}^{-1}(I_\mu\otimes
U_q(z)^{-1}
\]is a right inverse of $\cm(z)$. The inverses $U_q(z)^{-1}$ and $W_r(z)^{-1}$ can be computed easily,
since one has for instance $\widetilde{U}_q(z)=U_q(z)D_\mu$ and $\widetilde{U}_q(z)^{-1}=\widetilde{U}_q(-z)$.
The inverses $\cl_{\nu,r}^{-1}$ and $\cl_{\mu,q}^{-1}$ can be computed in view of
the formulas just above Theorem~\ref{thm:azfactor1}. Since $\cl_{\nu,r}^{-1}=\cl_{\nu,r}(0)^{-1}$,
one obtains that its $ij$-block ($i\geq j$) is given by ${i \choose j}(S_q^\top)^{i-j}(-1)^{i-j}$. Summarizing, we have
\begin{proposition}
The matrix $\cm(z)$ has a right inverse iff $\nu\geq q$ and $\mu\leq r$,
in which case a right inverse is
\[
\cm(z)^+=(I_\nu\otimes
W_r(z)^{-1})\cl_{\nu,r}(0)^{-1}\bar{\ca}^\top\cl_{\mu,q}(0)^{-1}(I_\mu\otimes
U_q(z)^{-1}.
\]
All right inverses form an affine space of dimension $(\nu r - \mu q)\mu q$.
\end{proposition}

\end{document}